\documentclass[12pt,twoside,english,a4paper]{article}
\usepackage{babel,amssymb,amsthm}
\usepackage{graphicx}
\usepackage{amsmath}
\usepackage{bm}
\usepackage{mathdots}

\newcommand{\smallfrac}[2]{{\textstyle\frac{#1}{#2}}} 
\newcommand{\jump}[1]{[\![#1]\!]}
\newcommand{\ave}[1]{\{\!\!\{#1\}\!\!\}}

\newcommand{\bs}{\boldsymbol}
\newcommand{\nn}{\mathbf n}
\newcommand{\rr}{\mathbf r}
\newcommand{\xx}{\mathbf x}
\newcommand{\bb}{\mathbf b}
\newcommand{\mm}{\mathbf m}
\newcommand{\nnt}{\widetilde{\mathbf n}}
\newcommand{\aaa}{\mathbf a}

\newtheorem{proposition}{Proposition}[section]

\newtheorem{lemma}[proposition]{Lemma}

\numberwithin{equation}{section}

\title{A fully discrete Calder\'on Calculus for the two-dimensional elastic wave equation}

\date{\today}
\author{V{\'\i}ctor Dom{\'\i}nguez\footnote{Departmento de Ingenier\'\i a Matem\'atica en Inform\'atica, Universidad P\'ublica de Navarra, 31500 Tudela, Spain {\tt victor.dominguez@unavarra.es}}, 
Tonatiuh S\'anchez-Vizuet \& Francisco--Javier
Sayas\footnote{Department of Mathematical Sciences, University of Delaware,
Newark, DE 19716, USA. {\tt \{tonatiuh,fjsayas\}@math.udel.edu}. Partially supported by NSF
grant DMS 1216356.}}

\begin{document}

\maketitle

\begin{abstract}
In this paper we present a full discretization of the layer potentials and boundary integral operators for the elastic wave equation on a parametrizable smooth closed curve in the plane. The method can be understood as a non-conforming Petrov-Galerkin discretization, with a very precise choice of testing functions by symmetrically combining elements on two staggered grids, and using a look-around quadrature formula. Unlike in the acoustic counterpart of this work, the kernel of the elastic double layer operator includes a periodic Hilbert transform that requires a particular choice of the mixing parameters. We give mathematical justification of this fact. Finally, we test the method on some frequency domain and time domain problems, and demonstrate its applicability on smooth open arcs.
\end{abstract}

\section{Introduction}

In this paper we give a simultaneous discretization of all integral operators that appear in the Calder\'on projector for the time-harmonic elastic wave equation on a smooth parametrizable curve in the plane. We give experimental evidence that the methods are of order three for all computed quantities in the boundary, and also for potential postprocessings. The paper can be considered as a continuation of  \cite{DoLuSa:2014b}, where a fully discrete Calder\'on calculus for the acoustic wave equation was developed. While in that case a one-parameter family of discretizations was defined, we will give mathematical evidence here that the family needs to be restricted to a single discretization method for elastodynamics. Because of the way the operators are discretized, we can take advantage of Lubich's Convolution Quadrature techniques \cite{Lubich:1988, Lubich:1994} and obtain discretization methods for transient problems in scattering of elastic waves.

Let us first give some context to this work. Quadrature methods for periodic integral (and pseudodifferential) equations appeared in the work of Jukka Saranen. In particular, in \cite{SaSc:1993}, it was discovered that logarithmic integral equations can be given a very simple treatment providing methods of order two with simple-minded discretization arguments, as long as some parameters were chosen in a particular (and not easy to justify) way. Related references are \cite{SaSl:1992}, \cite{SlBu:1992}, and \cite{CeDoSa:2002}. Exploiting these ideas, an equally simple quadrature method of order three for a system of integral equations, that combined the single layer and double layer operators, was given in \cite{DoRaSa:2008}. A recent article \cite{DoLuSa:TA} opened new ways by offering an extremely simple form of discretizing the hypersingular operator associated to the Helmholtz equation in a smooth closed curve. As a consequence, there was a chance of creating a full discretization of all operators for the Helmholtz equation \cite{DoLuSa:2014a}, using $\mathcal O(N^2)$ evaluations of the kernels and obtaining second order approximations for all unknowns in a wide collection of integral equations that could be discretized simultaneously. It was then in \cite{DoLuSa:2014b} where it was discovered that a symmetrization process led to order three discretizations. Consistency error estimates for the second order methods were given in \cite{DoLuSa:2014a}, taking advantage of already existing results. The consistency analysis of the order three methods is the subject of current research.

There are two reasons why the extension of the techniques of \cite{DoLuSa:2014b} to the realm of elastic waves is not straightforward. The first of them is the fact that the double layer operator and its adjoint contain a strongly singular operator (a perturbation of the periodic Hilbert transform) that makes the operators of the second kind much more difficult to handle. We show (with experiments and with rigorous mathematical justification) that the consistency error of our type of discretizations when applied to the periodic Hilbert transform has order three. The second difficult ingredient is the need for a regularized formula (\'a la N\'ed\'elec-Maue) that is compatible with our way of discretizing the other operators. While the regularized formula for static elasticity is easy to find in the literature \cite{Nedelec:1982}, the elastic case is much more involved. For instance, in \cite{ChKrMo:2000} the authors opt for a subtraction technique, where the regularized static hypersingular operator is  used and then the difference between the time-harmonic and the static operators is prepared for discretization. Here we will use a formula due to Frangi and Novati \cite{FrNo:1998}, which we fully develop so that our results can be easily replicated. For more literature on regularization of hypersingular operators in elastodynamics we refer to \cite{FrNo:1998}.

We want to emphasize that the one of the top features of these discretizations is their easiness. For elastic waves it is the collection of integral operators (in particular the hypersingular operator) that can be seen as threatening. However, once their formulas are available, the discretization follows very simple rules, with the simple mindedness of low order finite difference methods. Once again, the goal is not the discretization of one equation at a time, but of all possible associated integral equations at the same time. The same discretization methods can be applied verbatim to steady-state two dimensional elasticity and it is likely that they work for the Stokes system, given the fact that the Stokes layer potentials are a combination of the Laplace and Lam\'e potentials.

The paper is structured as follows. The remainder of this section serves as an introduction to the main concepts of the Calder\'on calculus for the linear elasticity problem in the frequency domain. In Section \ref{sec:2} we introduce the main discrete elements (sampling of the curve, mixing matrices) and give a general interpretation of the methods to be defined as non-conforming Petrov-Galerkin methods with very precisely chosen quadrature approximation. In Sections \ref{sec:3} and \ref{sec:4} we present respectively the discrete layer potentials and integral operators. In Section \ref{sec:5} (supported by some technical lemmas given in Appendix \ref{app:B}) we give mathematical justification of why the shape of the testing devices (that we named the fork and the ziggurat in \cite{DoLuSa:2014b}) provides order three consistency errors for the periodic Hilbert transform, which is the only integral operator that appears in the set of boundary integral equations for elasticity and not for acoustic waves. In Section \ref{sec:6} we present some experiments in the frequency domain. In Section \ref{sec:7} we show some experiments in the time-domain using Convolution Quadrature. Finally, in Section \ref{sec:8} we give a treatment of smooth open arcs using a double cosine sampling of the arc and repeating the same ideas of the previous sections.

\paragraph{Geometric entities.}
Let us start by fixing the geometric frame of this paper. A closed simple curve $\Gamma \subset \mathbb R^2$ separates a bounded domain $\Omega^-$ from its exterior $\Omega^+$. The curve will be given by a smooth 1-periodic parametrization $\mathbf x=(x_1,x_2):\mathbb R\to \Gamma$ satisfying $|\mathbf x'(t)|> 0$ for all $t$. It is also assumed that the normal vector field $\mathbf n:=(x_2',-x_1')$ (which will never be normalized in the sequel) points outwards, i.e., the parametrization has positive orientation. For a given smooth enough vector field $\boldsymbol U:\mathbb R^2\setminus\Gamma \to \mathbb R^2$, we will define its parametrized traces (in the style of \cite[Section 8.2]{Kress:2014})
\begin{equation}\label{eq:1.1}
\boldsymbol\gamma^\pm \boldsymbol U:=\boldsymbol U^\pm\circ\mathbf x:\mathbb R \to \mathbb R^2.
\end{equation}
If $\nabla\boldsymbol U$ is the Jacobian matrix of $\boldsymbol U$, we define the stress tensor as
\[
\boldsymbol\sigma(\boldsymbol U):=
	\mu (\nabla\boldsymbol U+(\nabla\boldsymbol U)^\top)+
	\lambda (\nabla\cdot\boldsymbol U) \mathbf I_2 \qquad 
	\mathbf I_2:=\left[\begin{array}{cc} 1 & 0 \\ 0 & 1 \end{array}\right].
\]
This tensor is defined in terms of the two positive Lam\'e parameters $\lambda$ and $\mu$, which describe the elastic properties of a homogeneous isotropic material. The parametrized normal tractions on both sides of $\Gamma$ are defined by
\begin{equation}
\boldsymbol t^\pm \boldsymbol U:=
	(\boldsymbol\sigma(\boldsymbol U)\circ\mathbf x)\mathbf n:\mathbb R \to \mathbb R^2.
\end{equation}
Note that the choice for a non-normalized normal vector field simplifies a boundary-parametrized version of Betti's formula
\[
\int_{\Omega^\pm} 
	\Big(\boldsymbol\sigma(\boldsymbol U): \nabla\boldsymbol V
	+(\mathrm{div}\,\boldsymbol\sigma(\boldsymbol U))\cdot \boldsymbol V\Big)=
	\mp \int_0^1 (\boldsymbol t^\pm\boldsymbol U)(\tau) \cdot (\boldsymbol\gamma^\pm \boldsymbol V)(\tau)
				\mathrm d\tau,
\]
where $\mathrm{div}$ is the divergence operator applied to the rows of a matrix-valued function and the colon denotes the matrix inner product related to the Frobenius norm. The two sided traces and normal tractions are combined in jumps and averages as follows:
\begin{eqnarray*}
\jump{\boldsymbol\gamma\boldsymbol U}:=
	\boldsymbol\gamma^-\boldsymbol U-\boldsymbol\gamma^+\boldsymbol U,
&&
\ave{\boldsymbol\gamma\boldsymbol U}:=
	\smallfrac12(\boldsymbol\gamma^-\boldsymbol U+\boldsymbol\gamma^+\boldsymbol U),\\
\jump{\boldsymbol t\boldsymbol U}:=
	\boldsymbol t^-\boldsymbol U-\boldsymbol t^+\boldsymbol U,
&&
\ave{\boldsymbol t\boldsymbol U}:=
	\smallfrac12(\boldsymbol t^-\boldsymbol U+\boldsymbol t^+\boldsymbol U).
\end{eqnarray*}

\paragraph{The elastic wave equation.} In addition to the Lam\'e parameters $\lambda$ and $\mu$, another positive parameter $\rho>0$ (density) will be needed to describe the material properties of the medium. The transient elastic wave equation in absence of forcing terms is given by
\begin{equation}\label{eq:1.3}
\rho\, \partial_t^2 \boldsymbol U=\mathrm{div}\,\boldsymbol\sigma(\boldsymbol U).
\end{equation}
This equation is usually complemented with initial conditions and, when set on unbounded domains, some kind of causality condition at infinity. Its Laplace transform is given by
\begin{equation}\label{eq:1.4}
\rho\,s^2\, \boldsymbol U=\mathrm{div}\,\boldsymbol\sigma(\boldsymbol U).
\end{equation}
(We will use the same symbol for the unknown of the time-domain and in the Laplace-domain, since we are not going to be mixing them until the very end of the paper.) The time-harmonic wave equation is equation \eqref{eq:1.4} with $s=-\imath\omega$, where $\omega$ is the frequency. The time-harmonic wave equation in an exterior domain like $\Omega^+$ requires the Kupradze-Sommerfeld radiation condition at infinity, see for instance \cite[Section 2.4.3]{AmKaLe:2009}.  When $s$ takes values in $\mathbb C_+:=\{ s\in \mathbb C\,:\, \mathrm{Re}\,s>0\}$, the radiation condition is substituted by demanding that $\boldsymbol U$ and $\nabla\boldsymbol U$ are square integrable.

We will find it advantageous to stay with the Laplace transform version \eqref{eq:1.4} even to describe the time-harmonic case. In practice this will eliminate imaginary units from fundamental solutions and integral operators, and it will impose the use of Macdonald functions (modified Bessel functions of the second kind) instead of Hankel functions. In any case, we will give a simple table of modifications needed to write the operators in the frequency domain. 

\paragraph{Potentials and operators of the Calder\'on Calculus.} Let us now consider two vector-valued 1-periodic functions (parametrized densities) $\boldsymbol\psi,\boldsymbol\eta:\mathbb R\to \mathbb C^2$, and let us assume that they are smooth enough for the next arguments to hold. For any $s\in \mathbb C_+$, the transmission problem
\begin{equation}\label{eq:1.6}
\rho \, s^2\, \boldsymbol U=\mathrm{div}\,\boldsymbol\sigma(\boldsymbol U) 
	\quad \mbox{in $\mathbb R^2\setminus\Gamma$},\qquad
\jump{\boldsymbol\gamma\boldsymbol U}=\boldsymbol\psi,\qquad
\jump{\boldsymbol t\boldsymbol U}=\boldsymbol\eta,
\end{equation} 
admits a unique solution in the Sobolev space $H^1(\mathbb R^2\setminus\Gamma)^2$, as follows from a simple variational argument. (The fact that the trace and normal traction are parametrized does not change the classical treatment of this equation.) The solution of this problem is clearly a linear function of the densities and will be written in the form
\begin{equation}\label{eq:1.6b}
\boldsymbol U=\mathbf S(s)\boldsymbol\eta-\mathbf D(s)\boldsymbol\psi.
\end{equation}
Note that \eqref{eq:1.6} has to be complemented with the radiation condition when $s=-\imath\omega$ for positive $\omega$. The operators $\mathbf S(s)$ and $\mathbf D(s)$ are respectively called the single and double layer elastic potentials. By definition,
\begin{equation}\label{eq:1.7}
\jump{\boldsymbol\gamma\mathbf S(s)\boldsymbol\eta}=\bs 0, \qquad 
\jump{\boldsymbol t\mathbf S(s)\boldsymbol\eta}=\bs\eta,\qquad
\jump{\bs\gamma\mathbf D(s)\bs\psi}=-\bs\psi,\qquad
\jump{\bs t\mathbf D(s)\bs\psi}=\bs 0,
\end{equation}
for arbitrary densities $\bs\eta$ and $\bs\psi$. We then define the four associated boundary integral operators
\begin{subequations}\label{eq:1.8}
\begin{alignat}{6}
\mathbf V(s)\bs\eta:=\ave{\bs\gamma\mathbf S(s)\bs\eta}, 
	&\qquad &
\mathbf J(s)\bs\eta:=\ave{\bs t\mathbf S(s)\bs\eta}, \\
\mathbf K(s)\bs\psi:=\ave{\bs\gamma\mathbf D(s)\bs\psi}, 
	&\qquad &
\mathbf W(s)\bs\psi:=-\ave{\bs t\mathbf D(s)\bs\psi}.
\end{alignat}
\end{subequations}
Explicit integral expressions of the potentials and operators \eqref{eq:1.6b} and \eqref{eq:1.8} will be given in Sections 
\ref{sec:3}-\ref{sec:4}. Given \eqref{eq:1.7}-\eqref{eq:1.8} we can easily derive the jump relations:
\begin{subequations}\label{eq:1.9}
\begin{alignat}{6}
\bs\gamma^\pm \mathbf S(s)=\mathbf V(s),
	&\qquad &
\bs t^\pm \mathbf S(s)=\mp\smallfrac12\mathbf I+\mathbf J(s),\\
\bs\gamma^\pm \mathbf D(s)=\pm\smallfrac12\mathbf I+\mathbf K(s),
	&\qquad&
\bs t^\pm \mathbf D(s)=-\mathbf W(s).
\end{alignat}
\end{subequations}
The operators \eqref{eq:1.8} receive the following respective names: single layer, transpose double layer, double layer, and hypersingular operators. It is important to note that the adjoint of $\mathbf K(s)$ is not $\mathbf J(s)$ but $\mathbf J(\overline s)$. We are not going to worry about the construction of well-posed integral equations for different boundary value problems for the elasticity equation \eqref{eq:1.4}. Instead, following \cite{DoLuSa:2014a}-\cite{DoLuSa:2014b}, we will simultaneously discretize the potentials \eqref{eq:1.6b}, the operators \eqref{eq:1.8}, and the two identity operators in \eqref{eq:1.9} in a way that is stable and compatible, so that all these elements can be used to build discrete counterparts of well-posed boundary integral equations for elastic waves. 

\section{Discrete elements and mixing operators}\label{sec:2}

\paragraph{Sources and observation points.}
The curve $\Gamma$ will be sampled three times, once for the location for densities (sources) and twice for the location of collocation points (targets). The sampling is exactly the same as in \cite{DoLuSa:2014b}. We choose a positive integer $N$, define $h:=1/N$ and sample midpoints, breakpoints, and normals on the main grid (sources):
\begin{equation}\label{eq:2.1}
\mathbf m_j:=\mathbf x(j\,h),\qquad
\mathbf b_j:=\mathbf x((j-\smallfrac12)\,h),\qquad
\mathbf n_j:=h\mathbf n(j\,h).
\end{equation}
This is done for $j\in \mathbb Z_N$, that is, the points are indexed modulo $N$. We then repeat the same process by shifting the uniform grid in parametric spaces $\pm h/6$:
\begin{equation}\label{eq:2.2}
\mathbf m_i^\pm:=\mathbf x((i\pm \smallfrac16)\,h),\qquad
\mathbf b_i^\pm:=\mathbf x((i-\smallfrac12\pm\smallfrac16)\,h),\qquad
\mathbf n_i^\pm:=h\mathbf n((i\pm\smallfrac16\,h).
\end{equation}
This is all the information that is needed from the parametric curve $\Gamma$. Once these elements are sampled, we do not need the parametrization of the curve any longer. It is important to emphasize that all the following constructions {\em can be easily extended to multiple scatterers.} The details are given in \cite{DoLuSa:2014b}. As explained in that reference, we only need to create a next-index counter. The choice of the $\pm1/6$ shifting parameter is justified in \cite[Section 4]{DoLuSa:2014b} by an argument that chooses the optimal shifting of a trapezoidal rule applied to periodic logarithmic operators. This idea can be traced back to \cite{SaSc:1993}. In short, if $\log_\#(t):=\log(4\sin^2(\pi t))$ is a periodic logarithmic function, then for a smooth enough 1-periodic function $\phi$
\begin{equation}\label{eq:2.BB}
\int_0^1 \log_\# (t-\tau) \, \phi(\tau)\mathrm d\tau
-h\sum_{j=1}^N \log_\#(t-j\,h) \phi(j\,h)=\mathcal O(h^2) \quad \Longleftrightarrow\quad
	t=(i\pm\smallfrac16)\,h.
\end{equation}
In order to introduce some matrices relevant to the method, we will use the following notation: $\mathrm{TC}(a_1,\ldots,a_N)$ is the Toeplitz Circulant matrix whose first row is the vector $(a_1,\ldots,a_N)$.

\paragraph{Three mixing matrices.} We first introduce a block-diagonal matrix with circulant tridiagonal blocks:
\begin{equation}\label{eq:2.3}
\mathrm Q:=
	\left[\begin{array}{cc} 
		\mathrm Q_{\mathrm{sc}} & 0 \\
		0 & \mathrm Q_{\mathrm{sc}}
	\end{array}\right],
\qquad
\mathrm Q_{\mathrm{sc}}:=\smallfrac1{24}\, \mathrm{TC}\left(22, 1,0,\ldots,0,1\right).
\end{equation}
The coefficients of the `scalar' operator $\mathrm Q_{\mathrm{sc}}$ are related to a look-around quadrature formula
\begin{equation}\label{eq:2.AA}
\int_{-\frac12}^\frac12 \phi(t)\mathrm d t\approx \frac1{24}(\phi(-1)+22\phi(0)+\phi(1)),
\end{equation}
which has degree of precision three.
We also need two matrices that will be used for averaging the information on the two observation grids. Both of them are block diagonal with circulant bidiagonal blocks:
\begin{equation}\label{eq:2.4}
\mathrm P^\pm:=
	\left[\begin{array}{cc} 
		\mathrm P^\pm_{\mathrm{sc}} & 0 \\
		0 & \mathrm P^\pm_{\mathrm{sc}}
	\end{array}\right],
\qquad
\mathrm P^+_{\mathrm{sc}}:=\smallfrac1{12}\,\mathrm{TC}(5,0,\ldots,0,1),
\qquad
\mathrm P^-_{\mathrm{sc}}:=(\mathrm P^+_{\mathrm{sc}})^\top.
\end{equation}
The matrices $\mathrm P^\pm_{\mathrm{sc}}$ are particular cases of a one-parameter dependent construction of numerical schemes in \cite{DoLuSa:2014b}. As we will justify in Section \ref{sec:5}, this is the right choice of parameters to sample correctly some periodic Hilbert transforms that appear in the double layer operator and its transpose. 

\paragraph{Discrete mass and differentiation matrices.} In this context, the identity operator will be approximated by the block-diagonal matrix
\begin{equation}\label{eq:mass}
\mathrm M:=
	\left[\begin{array}{cc} 
		\mathrm M_{\mathrm{sc}} & 0 \\
		0 & \mathrm M_{\mathrm{sc}}
	\end{array}\right],
\qquad
\mathrm M_{\mathrm{sc}}:=\smallfrac1{9}
	\mathrm{TC}(7,1,0,\ldots,0,1).
\end{equation}
The final matrix to be introduced in this section corresponds to differentiation:
\begin{equation}
\mathrm D:=
	\left[\begin{array}{cc} 
		\mathrm D_{\mathrm{sc}} & 0 \\
		0 & \mathrm D_{\mathrm{sc}}
	\end{array}\right],
\qquad
\mathrm D_{\mathrm{sc}}:=\mathrm{TC}(-1,1,0,\ldots,0).
\end{equation}

\paragraph{A formal explanation.} After parametrization (see the formulas in Sections \ref{sec:3} and \ref{sec:4}), potentials and integral operators for two-dimensional linear elasticity can be seen as acting on periodic functions and, in the case of the integral operators, outputting periodic functions. In the variational theory for boundary integral equations, traces take values in the Sobolev space $H^{1/2}(\Gamma)$ and normal stresses in its dual $H^{-1/2}(\Gamma)$. These spaces are used as trial spaces and also as test spaces. The main idea behind the methods in \cite{DoLuSa:2014b} (a paper devoted to acoustic waves) is using a non-conforming discretization for the spaces $H^{\pm1/2}$, which are the 1-periodic Sobolev spaces we obtain after parametrizing $\Gamma$ in $H^{\pm1/2}(\Gamma)$. Let now $\delta_t$ denote the periodic Dirac delta concentrated at the point $t$, or, better said, the Dirac comb concentrated at $t+\mathbb Z$. There are four kinds of elements in the discrete Calder\'on calculus:
\begin{itemize}
\item The unknowns that live in $H^{-1/2}$ are approximated by linear combinations of $\delta_{jh}$. 
\item Whenever $H^{-1/2}$ plays the role of a test space (i.e., the output of the operator is in $H^{1/2}$), we use a  testing device made of symmetric combinations of Dirac deltas in the companion meshes:
\begin{equation}\label{eq:2.8}
\frac{a}2 (\delta_{(i-1/6)h}+\delta_{(i+1/6)h})+\frac{1-a}2(\delta_{(i-5/6)h}+\delta_{(i+5/6)h}).
\end{equation}
As we have already mentioned, the choice of moving the grid $\pm1/6h$ is forced by the need to correctly sample the logarithmic kernel singularities \eqref{eq:2.BB}. By periodicity, the grids displaced $\pm 5/6h$ are the same as those displaced $\mp 1/6h$, and therefore, all Dirac deltas in the testing device \eqref{eq:2.8} satisfy the second order condition \eqref{eq:2.BB}. Third order is actually attained thanks to symmetry in \eqref{eq:2.8}. 
\item For unknowns in $H^{1/2}$, we use piecewise constant functions, spanned by the periodic characteristic functions defined by $\chi_j'=\delta_{(j-1/2)h}-\delta_{(j+1/2)h}$. 
\item When $H^{1/2}$ is used as a test space (i.e., the operator has output in $H^{-1/2}$), we test the equation by the piecewise constant functions $\xi_i$ defined by
\begin{eqnarray*}
\xi_i' &=& \frac{a}2 (\delta_{(i-1/2-1/6)h}+\delta_{(i-1/2+1/6)h})+\frac{1-a}2(\delta_{(i-1/2-5/6)h}+\delta_{(i-1/2+5/6)h})\\
       &&-\frac{a}2 (\delta_{(i+1/2-1/6)h}+\delta_{(i+1/2+1/6)h})-\frac{1-a}2(\delta_{(i+1/2-5/6)h}+\delta_{(i+1/2+5/6)h}).
\end{eqnarray*}
\end{itemize}
Whenever a Dirac delta acts on a kernel (as input or as a test), this one is automatically evaluated and there is no need for further approximation. Nevertheless, if a piecewise constant function is used, there are still integrals to be approximated. These are decomposed as combinations of integrals of the kernels on intervals of length $h$ and then approximated using the look-around formula \eqref{eq:2.AA}. As shown in \cite{DoLuSa:2014b}, all values of $a\in [1/2,1]$ give order three methods valid for the acoustic Calder\'on calculus. The choice $a=1$ gives the simplest test functions. The choice $a=5/6$ gives the linear combinations that produce the mixing matrices \eqref{eq:2.4}. The corresponding testing devices were named the fork and the ziggurat (see Figure \ref{figure:fork}). We will see in Section \ref{sec:5} that this choice is the only one that provides order-three discretizations of the Hilbert transform-style operators that appear in elastodynamics and not in scalar waves.

\begin{figure}[h]
\begin{center}
\includegraphics[width=5cm]{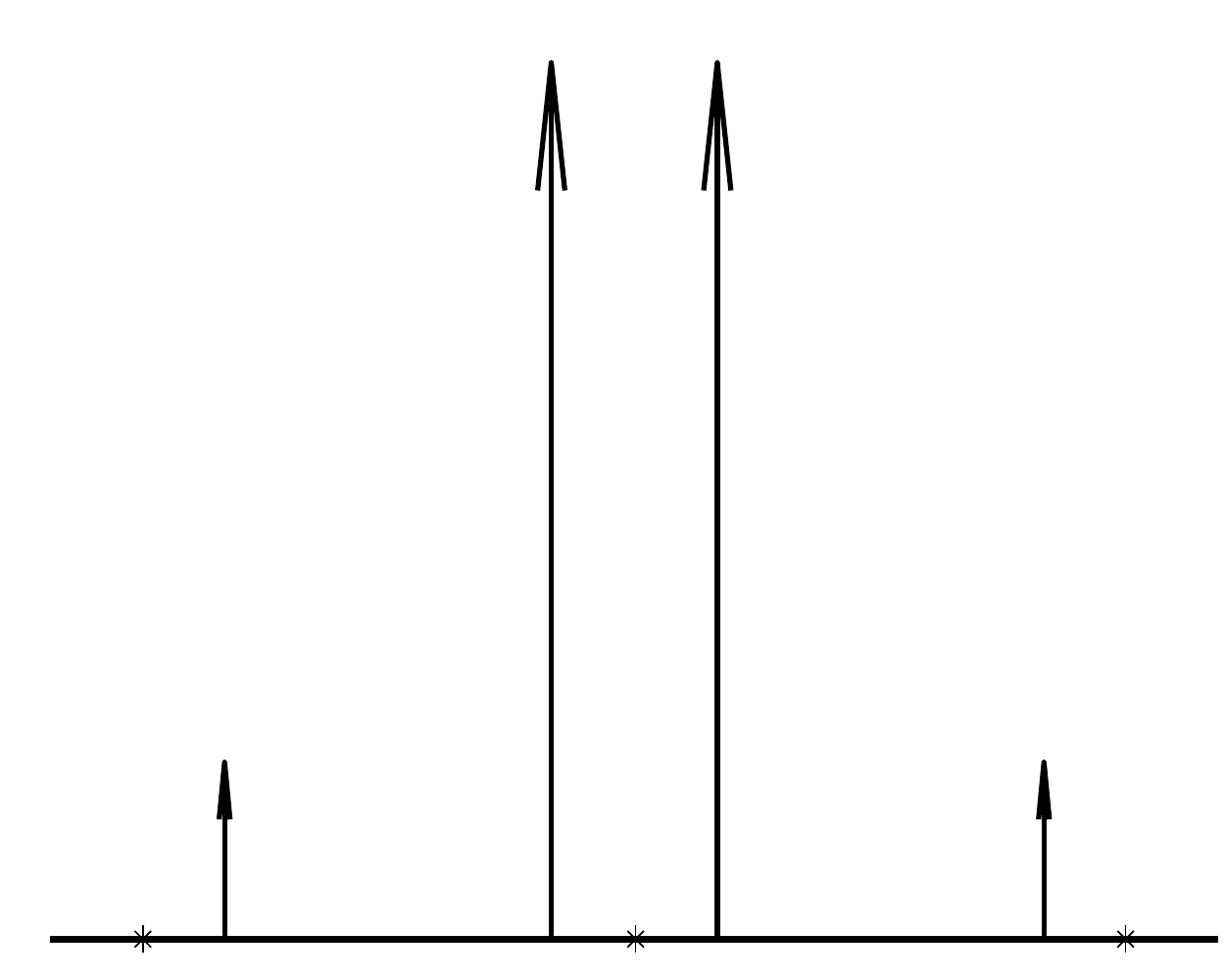}
\qquad
\includegraphics[width=5cm]{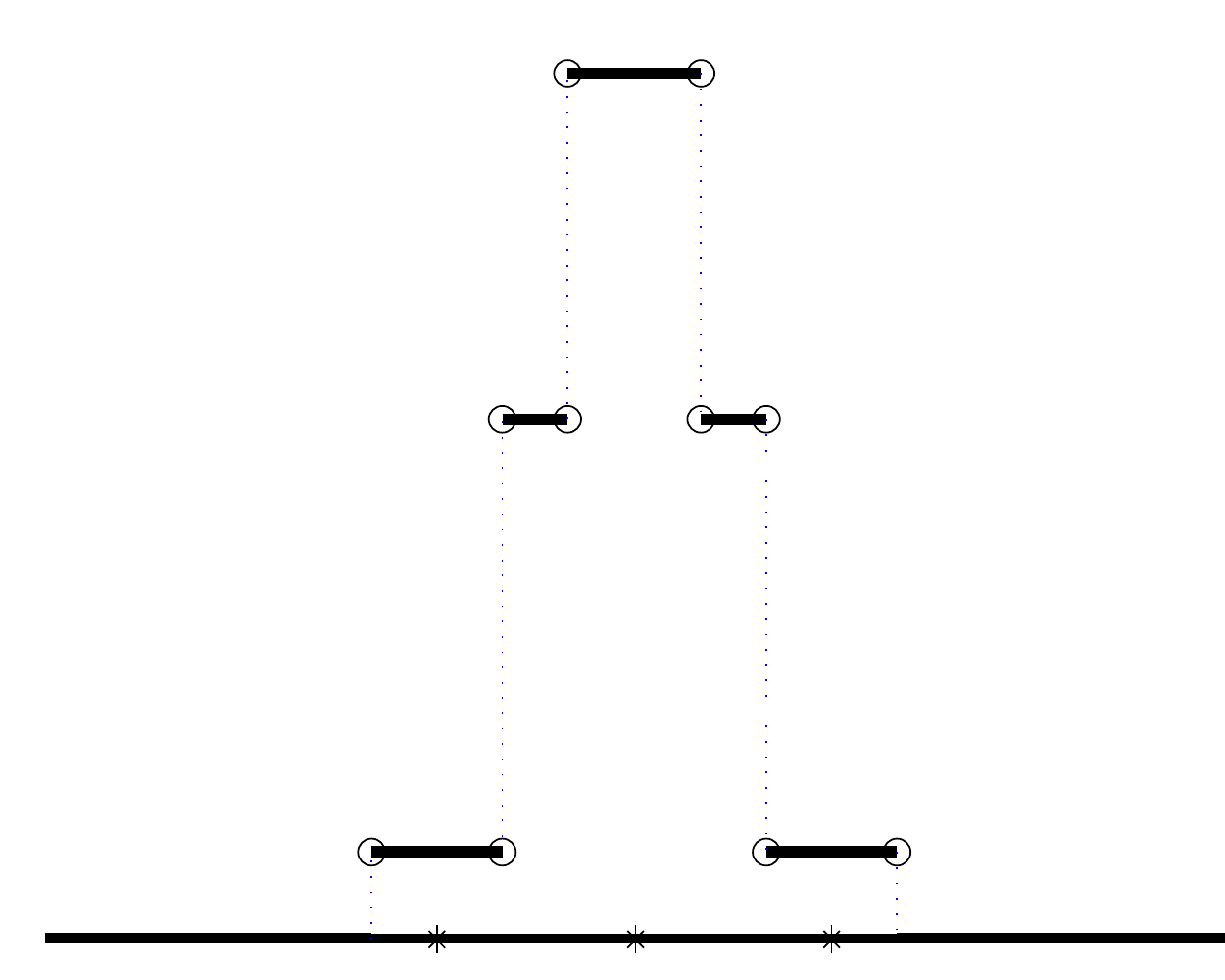}
\end{center}
\caption{A sketch of the shapes of the fork and the ziggurat, the main testing devices of the fully discrete Calder\'on calculus}\label{figure:fork}
\end{figure}


\section{Layer potentials}\label{sec:3}

\paragraph{Notation for blocks.} Densities will be discretizations of 1-periodic vector fields. At the discrete level, these will become vectors in $\mathbb C^{2N}\equiv \mathbb C^N\times \mathbb C^N$. In principle, we will think of discrete densities as a vector of $2N$ entires , where the first $N$ entries discretize the first component of the continuous vector field, and are followed by $N$ entries for the second component. 
Consider then an operator $\mathbb C^{2N}\to \mathbb C^{2M}$, given through left-multiplication by a $(2M)\times (2N)$ matrix, say $\mathrm A$. This matrix can be decomposed in four $M\times N$ blocks. Alternatively, in can be presented by $MN$ blocks of $2\times 2$ size. Thus, we will write
\[
\mathrm A_{ij}:=
	\left[\begin{array}{cc}
		a_{ij} & a_{i,j+N}\\
		a_{i+M,j} & a_{i+M,j+N}
	\end{array}\right]	
\]
to make a simpler transition from the continuous expressions to the discrete ones. 

\paragraph{More physical parameters and auxiliary functions.} In the expression for the fundamental solution of equation \eqref{eq:1.4}, the following parameters will be relevant:
\begin{equation}\label{eq:3.A}
c_L:=\sqrt{\frac{\lambda+2\mu}\rho}, \qquad c_T:=\sqrt\frac{\mu}{\rho},
\qquad \xi:=\frac{c_T}{c_L}=\sqrt{\frac{\mu}{\lambda+2\mu}}.
\end{equation}
The quantity $c_L$ is the speed of pressure (or longitudinal) waves, while $c_T$ is the speed of shear (or transverse) waves. We will make ample use of the modified Bessel function of the second kind and order $n$, $K_n$, for non-negative integer $n$. For readers who want to translate from the Laplace-domain notation to frequency domain notation, recall that $s=-\imath \omega$ and
\[
K_n(s)=\frac{\pi\imath}2 e^{\frac{n\pi\imath}2} H^{(1)}_n (\imath s)
=\frac{\pi\imath}2 e^{\frac{n\pi\imath}2} H^{(1)}_n (\omega).
\]
The following two functions
\begin{subequations}\label{eq:3.0}
\begin{eqnarray}
\psi(r) &:=& K_0(r/c_T)+\frac{c_T}r \left( K_1(r/c_T)-\xi K_1(r/c_L)\right),\\
\chi(r) &:=& K_2(r/c_T)-\xi^2 K_2(r/c_L)
\end{eqnarray}
\end{subequations}
and their first derivatives will also be used.

\paragraph{The single layer potential.} The fundamental solution of the elastic wave equation \eqref{eq:1.4} is given by:
\begin{equation}\label{eq:3.1}
\bs E(\mathbf r;s):=
	\frac1{2\pi\mu}\left( \psi(s\,r) \mathbf I_{2}-\frac{\chi(s\,r)}{r^2} \mathbf r\otimes \mathbf r\right), 
\quad r:=|\mathbf r|, 
\quad\mathbf r\otimes\mathbf r:=\left[\begin{array}{cc} r_1^2 & r_2r_1 \\ r_1r_2 & r_2^2\end{array}\right].
\end{equation}
The parametrized single layer potential is given by
\begin{equation}\label{eq:3.2}
\mathbf S(s;\mathbf z)\bs\eta=\int_0^1 \bs E(\mathbf z-\mathbf x(\tau);s)\boldsymbol\eta(\tau)\mathrm d\tau,\qquad
\mathbf z\in \mathbb R^2\setminus\Gamma,
\end{equation}
for a (1-periodic) density $\bs\eta:\mathbb R\to \mathbb C^2.$ A discrete density $\bs\eta\in \mathbb C^{2N}$ (or, in functional notation, $\bs\eta:\mathbb Z_N \to \mathbb C^2$) can be presented by pairs of entries
\[
\bs\eta_j:=\left[\begin{array}{c} \eta_j \\ \eta_{j+N}\end{array}\right].
\] 
The associated discrete single layer potential is then defined by
\begin{equation}\label{eq:3.3}
\mathrm S_h(s;\mathbf z)\bs\eta:=\sum_{j=1}^N \bs E(\mathbf z-\mathbf m_j;s) \bs\eta_j\qquad 
\mathbf z\in \mathbb R^2\setminus\Gamma.
\end{equation}

\paragraph{The double layer potential.} 
The kernel function for the elastic double layer potential
\begin{equation}
\mathbf D(s;\mathbf z)\bs\psi=\int_0^1 \bs T(\mathbf z-\mathbf x(\tau),\nn(\tau);s)
	\bs\psi(\tau)\mathrm d\tau,\qquad
\mathbf z\in \mathbb R^2\setminus\Gamma,
\end{equation} 
is the function
\begin{eqnarray}
\nonumber
\bs T(\mathbf r,\mathbf n; s)
	&:=& -\frac{s\,\psi'(s\,r)}{2\pi r}
			\left((\mathbf r\cdot\mathbf n)\mathbf I_2+\mathbf n\otimes \mathbf r
				+\frac\lambda\mu \mathbf r\otimes \mathbf n\right)\\
\label{eq:3.4}
	& & +\frac1\pi\chi(s\,r) 
			\left( -\frac{4\mathbf r\cdot\mathbf n}{r^4} \mathbf r\otimes \mathbf r
				+\frac1{r^2} \mathbf n\otimes \mathbf r
				+\frac{\mathbf r\cdot\mathbf n}{r^2} \mathbf I_2 
				+\frac{1}{\xi^2 r^2}\mathbf r\otimes \mathbf n\right)\\
\nonumber
	& & +\frac{s}{2\pi}\chi'(s\,r)
			\left( \frac{2\mathbf r\cdot\mathbf n}{r^3}\,\mathbf r\otimes\mathbf r
				+\frac{\lambda}{\mu\,r}\mathbf r\otimes \mathbf n\right).
\end{eqnarray}
The discrete double layer operator is thus related to the vector fields $\bs T(\mathbf z-\mathbf m_j,\mathbf n_j;s)$
but, instead of this kernel being used as in \eqref{eq:3.3}, it requires to plug the mixing matrix $\mathrm Q$ of \eqref{eq:2.3} between the kernel and the density (recall the arguments given under the epigraph {\em A formal explanation} in Section \ref{sec:2}):
\begin{equation}\label{eq:3.6}
\mathrm D_h(s;\mathbf z)\bs\psi:=
\sum_{j=1}^N \bs T(\mathbf z-\mathbf m_j, \mathbf n_j;s)\bs\psi^{\mathrm{eff}}_j, \qquad 
\bs\psi^{\mathrm{eff}}:=\mathrm Q\bs\psi.
\end{equation}

\paragraph{Sampling of incident waves.} Let $\boldsymbol U^{\mathrm{inc}}$ be an incident wave, be it a plane pressure or shear wave, or a cylindrical wave. We first sample it and its associated normal traction on both companion grids, creating four vectors ($\boldsymbol\beta_0^\pm \in \mathbb C^{2N}$ and $\boldsymbol\beta_1^\pm \in \mathbb C^{2N}$) with blocks
\begin{equation}\label{eq:3.7}
\bs\beta_{0,i}^\pm:=\bs U^{\mathrm{inc}}(\mm_i^\pm),
\qquad
\bs\beta_{1,i}^\pm :=(\bs\sigma(\bs U^{\mathrm{inc}})(\mm_i^\pm) \big)\nn_i^\pm.
\end{equation}
The final version of the sampled incident wave involves the mixing operators the Section \ref{sec:2}:
\begin{equation}\label{eq:3.8}
\bs\beta_0:=\mathrm P^+\bs\beta_0^++\mathrm P^-\bs\beta_0^-,
\qquad
\bs\beta_1:=\mathrm Q(\mathrm P^+\bs\beta_1^++\mathrm P^-\bs\beta_1^-).
\end{equation}

\section{Boundary integral operators}\label{sec:4}

\paragraph{Three integral operators.}
The parametrized integral expressions for the operators $\mathbf V(s)$, $\mathbf K(s)$, and $\mathbf J(s)$ in \eqref{eq:1.8} use the fundamental solution \eqref{eq:3.1} and the kernel \eqref{eq:3.4}:
\begin{subequations}
\begin{alignat}{6}
(\mathbf V(s)\bs\eta)(t) 
	&:=\int_0^1 \bs E(\xx(t)-\xx(\tau);s)\bs\eta(\tau)\mathrm d\tau,\\
(\mathbf K(s)\bs\psi)(t)
	&:=\int_0^1 \bs T(\xx(t)-\xx(\tau),\nn(\tau);s)\bs\psi(\tau)\mathrm d\tau,\\
(\mathbf J(s) \bs\eta)(t)
	&:=\int_0^1 \bs T(\xx(\tau)-\xx(t),\nn(t);s)^\top \bs\eta(\tau) \mathrm d\tau.
\end{alignat}
\end{subequations}
Discretization is carried out in two steps. In a first step, two one-sided operators are produced for each operator:
\begin{subequations}\label{eq:4.2}
\begin{eqnarray}
\mathrm V_{ij}^\pm(s) &:=& \bs E(\mathbf m_i^\pm-\mathbf m_j;s),\\
\mathrm K_{ij}^\pm(s) &:=& \bs T(\mathbf m_i^\pm-\mathbf m_j,\mathbf n_j;s)\\
\mathrm J_{ij}^\pm(s) &:=&\bs T(\mathbf m_j-\mathbf m_i^\pm,\mathbf n_i^\pm;s)^\top.
\end{eqnarray}
\end{subequations}
Then the operators are mixed using the matrices \eqref{eq:2.3} and \eqref{eq:2.4}:
\begin{subequations}
\begin{eqnarray}
\mathrm V_h(s) &:=& \mathrm P^+\mathrm V^+_h(s)+\mathrm P^-\mathrm V^-_h(s),\\
\mathrm K_h(s) &:=& (\mathrm P^+\mathrm K^+_h(s)+\mathrm P^-\mathrm K^-_h(s))\mathrm Q,\\
\mathrm J_h(s) &:=& \mathrm Q(\mathrm P^+\mathrm J^+_h(s)+\mathrm P^-\mathrm J^-_h(s)).
\end{eqnarray}
\end{subequations}
Here $\mathrm V_h^\pm(s)$, $\mathrm K_h^\pm(s)$, and $\mathrm J_h^\pm(s)$ are the matrices defined in \eqref{eq:4.2}.
The logic for the placement of $\mathrm Q$ in the above formulas can be intuited using the non-conforming Petrov-Galerkin interpretation of the method given at the end of Section \ref{sec:2} (see also \cite{DoLuSa:2014b}). 

\paragraph{An integro-differential form for $\mathbf W$.}
The hypersingular operator $\mathbf W(s)$ will be discretized after using a decomposition stemming from a regularization formula in \cite{FrNo:1998}. In order to write it concisely we need to introduce some notation. The function
\begin{equation}
G(r):=\frac1{2\pi\rho}(K_0(r/c_T)-K_0(r/c_L))
\end{equation}
plays a key role in the formula. Four functions are derived from it:
\begin{subequations}\label{eq:4.5}
\begin{eqnarray}
\label{eq:5.1a}
G_1(r) &:=&\left(\frac1r \frac{\mathrm d}{\mathrm dr}+\frac{\mathrm d^2}{\mathrm d r^2}\right) G(r),\\
\label{eq:5.1b}
G_2(r) &:=& \left(\frac1r \frac{\mathrm d}{\mathrm dr}+\frac{\mathrm d^2}{\mathrm d r^2}\right) G_1(r)
		= \left(\frac{\mathrm d^4}{\mathrm dr^4}+\frac2r \frac{\mathrm d^3}{\mathrm dr^3}
			-\frac1{r^2} \frac{\mathrm d}{\mathrm dr^2}+\frac1{r^3} \frac{\mathrm d}{\mathrm dr}\right) G(r),\\
A(r) &:=&\frac1{r^2} G''(r)-\frac1{r^3} G'(r)
		=\frac1{r^2}\left(\frac{\mathrm d^2}{\mathrm dr^2}-\frac1r\frac{\mathrm d}{\mathrm dr}\right)G(r),\\
B(r) &:=& \frac1{r}G'(r)
\end{eqnarray}
\end{subequations}
Note that the differential operators in \eqref{eq:5.1a} and \eqref{eq:5.1b}
are the radial part of the two dimensional Laplacian and bi-Laplacian respectively. We will also use
the matrix-valued functions
\begin{subequations}
\begin{eqnarray}
\boldsymbol H(\mathbf r;s)&:=&s^2 A(s\,r) \mathbf r\otimes \mathbf r+ B(s\,r)\boldsymbol I_2, 
	\qquad r:=|\mathbf r|\\
\boldsymbol M(\rr,\nn,\nnt) &:=& (\nnt\otimes \nn)(\rr\otimes \rr)+(\rr\otimes \rr)(\nnt\otimes \nn),\\
M(\rr,\nn,\nnt) &:=& (\nnt\otimes \nn) : \rr.
\end{eqnarray}
\end{subequations}
Translating carefully the formulas in \cite{FrNo:1998}, we can write the hypersingular operator for the elastic wave equation in integro-differential (regularized) form
\begin{alignat}{6}\label{eq:4.7}
(\boldsymbol W(s)\boldsymbol\psi)(t):=&
	-\frac{\mathrm d}{\mathrm dt} \int_0^1 \boldsymbol W_0(\xx(t)-\xx(\tau);s) \frac{\mathrm d}{\mathrm d\tau}
										\boldsymbol\psi(\tau) \mathrm d\tau \\
\nonumber
	& + \int_0^1 \boldsymbol W_1(\xx(t)-\xx(\tau),\nn(t),\nn(\tau);s) \boldsymbol \psi(\tau)\mathrm d\tau.
\end{alignat}
where
\[
\boldsymbol W_0(\rr;s) :=
	4\mu^2 G_1(s r)-\boldsymbol H(\rr;s)
\]
and
\begin{eqnarray*}
\boldsymbol W_1(\rr,\nn,\nnt;s)	
	&:=&\frac{\lambda+2\mu}{\lambda+\mu} s^2 \Bigg(
				 \mu G_2(sr) \Big( \lambda \widetilde{\mathbf n}\otimes \mathbf n+
									\mu  \mathbf n\otimes \widetilde{\mathbf n}
										+ \mu (\widetilde{\mathbf n}\cdot\mathbf n) \, \boldsymbol I_2 \Big)\\
	& &   -\frac1{c_L^2} \Big(
			\lambda^2 G_1(sr) \widetilde{\mathbf n}\otimes \mathbf n \\
	& &   \phantom{-\frac1{c_L^2} \Big(} 
				+s^2 A(sr)
				\Big( 2\lambda\mu \boldsymbol M(\rr,\nn,\nnt) + \mu^2 M(\rr,\nn,\nnt) \boldsymbol I_2
				+\mu^2 \boldsymbol M(\rr,\nnt,\nn)\Big)\\
     &&  \phantom{-\frac1{c_L^2} \Big(} 
				+ B(sr) \Big( 4\lambda\mu \nnt\otimes \nn + \mu^2 (\nnt\cdot\nn )\boldsymbol I_2
							+\mu^2 s \nn\otimes \nnt\Big)\\
    &&   \phantom{-\frac1{c_L^2} \Big(}
				+ \mu^2 (\nnt\cdot\nn) \boldsymbol H(\rr;s) \Big)\Bigg).
\end{eqnarray*}
For discretization, we separate the two integral operators that appear in \eqref{eq:4.7}, building the blocks
\begin{subequations}
\begin{eqnarray}
\mathrm W^\pm_{ij,0}(s)
	&:=& \boldsymbol W_0(\bb_i^\pm-\bb_j;s),\\
\mathrm W^\pm_{ij,1}(s)
	&:=& \boldsymbol W_1(\mm_i^\pm-\mm_j,\nn_i^\pm,\nn_j;s),
\end{eqnarray}
\end{subequations}
and then mix them using the expression
\begin{equation}
\mathrm W_h(s) :=
	\mathrm D (\mathrm P^+\mathrm W_{0,h}^+(s)+\mathrm P^-\mathrm W_{0,h}^-(s))\mathrm D^\top
	 +\mathrm Q(\mathrm P^+\mathrm W_{1,h}^+(s)+\mathrm P^-\mathrm W_{1,h}^-(s))\mathrm Q.
\end{equation}

\section{Discrete treatment of Hilbert transforms}\label{sec:5}

In this section we justify the choice of parameters that are implicit to the definition of the mixing matrices \eqref{eq:2.4}. We recall that the paper \cite{DoLuSa:2014b} allowed for a one-parameter dependent family of test functions. The choice that produced the fork-and-ziggurat testing elements is the only one that works for the elasticity problem. The main element that appears in the elasticity integral operators $\mathbf K$ and $\mathbf J$ and does not appear in their acoustic counterparts is the periodic Hilbert transform:
\begin{equation}
(H\phi)(t):=\mathrm{p.v.}\int_0^1 \cot (\pi(t-\tau))\phi(\tau)\mathrm d \tau.
\end{equation}
(We note that it is customary to multiply this transform by the imaginary unit to relate it to the Hilbert transform on a circle.) For properties of this operator, we refer the reader to \cite[Section 5.7]{SaVa:2002}. We will use that the periodic Hilbert transform commutes with differentiation, i.e., $(H\phi)'=H\phi'$, which follows from an easy integration by parts argument. A key result related to trapezoidal approximation of the Hilbert transform is given next. It is an easy consequence of a Lemma \ref{lemma:B.2} proved in Appendix \ref{app:B}.

\begin{proposition}\label{prop:5.1}
Let $t_j:=j\,h$ with $h:=1/N$ and consider the following subspaces of trigonometric polynomials
\[
\mathbb T_h:=\mathrm{span}\{ e_n\,:\, -N/2\le n<N/2\}, \qquad e_n(t):=\exp(2\pi\imath n t). 
\]
Then
\begin{equation}\label{eq:5.2}
h \sum_{j=1}^N \cot(\pi(t-t_j))\phi(t_j)= (H \phi)(t)+\cot(\pi t/h) \phi(t) \qquad \forall \phi\in \mathbb T_h.
\end{equation}
\end{proposition}

Note that both sides of \eqref{eq:5.2} blow up when $t/h\in \mathbb Z$. We now specialize this formula to $t=t_i^\pm:=(i\pm\frac16)h$. For simplicity we will write $\alpha=\cot(\pi/6)=\sqrt3$. Using a density argument and Proposition \ref{prop:5.1}, it is easy to see that for $\phi\in \mathcal D:=\{ \phi\in \mathcal C^\infty(\mathbb R)\,:\,\phi(1+\cdot)=\phi\}$, we have
\begin{equation}\label{eq:5.3}
h \sum_{j=1}^N \cot(\pi(t_i^\pm-t_j))\phi(t_j)= (H \phi)(t_i^\pm)\pm\alpha \phi(t_i^\pm)+\mathcal O(h^m)\qquad \forall m\in \mathbb Z. 
\end{equation}
In \eqref{eq:5.3} and in the sequel, the Landau symbol will be used with the following precise meaning: $a_i=b_i+\mathcal O(h^m)$ denotes the existence of $C>0$ independent of $h$ such that that $\max_i |a_i-b_i|\le C h^m$.
Motivated by the quadrature formula \eqref{eq:2.AA}, we introduce the averaging operator
\begin{equation}
\Delta_h\phi:=\smallfrac1{24}\phi(\cdot-h)+\smallfrac{11}{12}\phi+\smallfrac1{24}\phi(\cdot+h),
\end{equation}
and note that
\begin{equation}\label{eq:5.5}
(\Delta_h\phi)(t_j)=\phi(t_j)+\frac{h^2}{24}\phi''(t_j)+\mathcal O(h^4) \qquad \forall \phi\in \mathcal D.
\end{equation}
Additionally, we consider the space of periodic piecewise constant functions
\[
S_h:= \{ \phi_h:\mathbb R\to \mathbb R\,:\, 
		\phi_h(1+\cdot)=\phi_h, \quad 
		\phi_h|_{(t_i-h/2,t_i+h/2)}\in \mathbb P_0\quad\forall i\}
\]
and the operator $D_h:\mathcal D \to S_h$ given by
\begin{equation}\label{eq:5.6}
D_h\phi\in S_h, \qquad \widehat{D_h\phi}(\mu)=\widehat\phi(\mu), \qquad -N/2\le \mu< N/2.
\end{equation}
This operator is well defined since periodic piecewise constant functions on a regular grid can be determined by any sequence of $N$ consecutive Fourier coefficients. (This fact seems to be known since at least the 30s of the past century \cite{QuCo:1938}, as mentioned in \cite{ArWe:1985}.) In particular, this operator has excellent approximation properties in a wide range of negative Sobolev norms (see \cite[Theorem 5.1]{Arnold:1983}). In \cite[Corollary 3.2]{DoSa:2001a} it is proved that
\begin{equation}\label{eq:5.7}
(D_h\phi)(t_j)=\phi(t_j)-\frac{h^2}{24}\phi''(t_j)+\mathcal O(h^4)\qquad \forall \phi\in \mathcal D.
\end{equation}
The following approximations of the periodic Hilbert transform
\begin{equation}\label{eq:5.8}
(H_h\phi)(t):=h\sum_{j=1}^N \cot(\pi(t-t_j))(\Delta_h\phi)(t_j).
\end{equation}
are relevant for our method, since for $\phi_h\in S_h$, we can write and approximate with \eqref{eq:2.AA}:
\begin{eqnarray*}
(H\phi_h)(t)&=&\sum_{j=1}^N \phi_h(t_j)\mathrm{p.v.}\int_{t_j-h/2}^{t_j+h/2}\cot(\pi(t-\tau))\mathrm d\tau \\
	&\approx & h\sum_{j=1}^N \phi_h(t_j) (\Delta_h\cot(\pi(t-\cdot)))(t_j) = (H_h\phi_h)(t), \qquad t/h\not\in\mathbb Z.
\end{eqnarray*}
 This means that using piecewise constant functions as trial functions, and following the idea of using look-around quadrature for integrals over intervals of length $h$ leads naturally to the operator $H_h$.
Note now that
\begin{equation}\label{eq:5.9}
\max_j|\cot(\pi(t_i^\pm-t_j))|=\mathcal O(h^{-1}).
\end{equation}
Therefore
\begin{alignat*}{6}
(H_hD_h\phi)(t_i^\pm)
	=& (H_h\phi)(t_i^\pm)-\frac{h^2}{24}(H_h \phi'')(t_i^\pm)+\mathcal O(h^3) 
		&\qquad &\mbox{(by \eqref{eq:5.7} and \eqref{eq:5.9})}\\
	=& h\sum_{j=1}^N \cot(\pi (t_i^\pm-t_j)) \phi(t_j)+\mathcal O(h^3)
	      &\qquad &\mbox{(by \eqref{eq:5.5} and \eqref{eq:5.9})}\\
	=& (H\phi)(t_i^\pm)\pm\alpha \phi(t_i^\pm)+\mathcal O(h^3)
		&\qquad &\mbox{(by \eqref{eq:5.3})},
\end{alignat*}
which, using Taylor expansions, implies that
\[
\smallfrac12 (H_hD_h\phi)(t_i^-)+\smallfrac12 (H_hD_h\phi)(t_i^+)=
\smallfrac12 (H\phi)(t_i^-)+\smallfrac12 (H\phi)(t_i^+)
		+\smallfrac16 h \alpha \phi'(t_i)+\mathcal O(h^3).
\]
Similarly
\[
\smallfrac12 (H_hD_h\phi)(t_{i+1}^-)+\smallfrac12 (H_hD_h\phi)(t_{i-1}^+)
	=\smallfrac12 (H\phi)(t_{i+1}^-)+\smallfrac12 (H\phi)(t_{i-1}^+)
	-\smallfrac56 h \alpha \phi'(t_i)+\mathcal O(h^3).
\]
If we then denote (see \eqref{eq:2.8})
\begin{equation}
\langle \varphi,\delta_i^\star\rangle:=\frac{a}2 (\varphi(t_i^-)+\varphi(t_i^+))
	+\frac{1-a}2(\varphi(t_{i-1}^+)+\varphi(t_{i+1}^-))
\end{equation}
it follows that
\[
\langle H_hD_h\phi, \delta_i^\star\rangle=\langle H \phi,\delta_i^\star\rangle+
\alpha h \frac{6a-5}6 \phi'(t_i)+\mathcal O(h^3),
\]
which shows that $a=5/6$ is the only choice that provides third order approximation.

\section{Experiments in the frequency domain}\label{sec:6}

In this section we show two numerical experiments, based on different integral equations, for the interior Dirichlet problem. The domain is the ellipse $(\frac{x}4)^2+(\frac{y}3)^2=1$ and we choose $\lambda=5$, $\mu=3$, and $\rho=2.5$ as physical parameters. We fix the wave number to be $k=3$ (so $s=-3\imath$ in our Laplace-domain based notation) and take data so that the exact solution is the sum of a pressure and a shear wave:
\begin{equation}\label{eq:6.0}
\bs U(\mathbf z):=
	\left(e^{\imath k (\mathbf z\cdot \mathbf d) /c_L}+
		e^{\imath k (\mathbf z\cdot \mathbf d^\perp) /c_T}\right)\mathbf d,
\qquad \mathbf d=(\smallfrac1{\sqrt2},\smallfrac1{\sqrt2}),
\qquad \mathbf d^\perp=(-\smallfrac1{\sqrt2},\smallfrac1{\sqrt2}),
\end{equation}
with $c_L$ and $c_T$ defined in \eqref{eq:3.A}. The solution will be observed in ten interior points $\mathbf z_i^{\mathrm{obs}}$, $i=1,\ldots,10$, randomly chosen in the circle $x^2+y^2=4$. The Dirichlet data is sampled using 
\eqref{eq:3.7}-\eqref{eq:3.8} to a vector $\boldsymbol\beta_0$. We are going to use a direct formulation, where the discrete elastic wave-field is given by the representation
\begin{equation}\label{eq:6.1}
\bs U_h (\mathbf z)=\mathrm S_h(s;\mathbf z)\bs\lambda-\mathrm D_h(s;\mathbf z)\bs\varphi.
\end{equation}
Since we are dealing with the Dirichlet problem, the computation of $\bs\varphi$ is rather straightforward: we just need to solve the sparse linear system (recall \eqref{eq:mass}),
\begin{equation}\label{eq:6.2}
\mathrm M \bs\varphi=\bs\beta_0,
\end{equation} 
which is then postprocessed to create the effective density (see \eqref{eq:3.6})
\begin{equation}\label{eq:6.3}
\bs\varphi^{\mathrm{eff}}:=\mathrm Q\bs\varphi.
\end{equation}
The vector $\bs\lambda$ (which approximates the normal traction) will be computed using a first kind discrete integral equation (based on the first integral identity provided by the Calder\'on projector)
\begin{equation}\label{eq:6.4}
\mathrm V_h(s)\bs\lambda=(\smallfrac12\mathrm M+\mathrm K_h(s))\bs\varphi,
\end{equation}
or a second kind integral equation (based on the second integral identity)
\begin{equation}\label{eq:6.5}
(-\smallfrac12\mathrm M+\mathrm J_h(s))\bs\lambda=-\mathrm W_h(s)\bs\varphi.
\end{equation}
We compute the following errors:
\begin{subequations}\label{eq:6.6}
\begin{eqnarray}
E_h^{\bs U} &:=& 
	\frac{\max_{i=1}^{10} | \bs U_h(\mathbf z_i^{\mathrm{obs}})-\bs U(\mathbf z_i^{\mathrm{obs}})|}
		{\max_{i=1}^{10} | \bs U(\mathbf z_i^{\mathrm{obs}})|},\\
E_h^{\bs\lambda} &:=& 
	\frac{\max_{i=1}^N |\bs\lambda_i-\bs\sigma(\bs U)(\mm_i)\cdot\nn_i|}
		{\max_{i=1}^N |\bs\sigma(\bs U)(\mm_i)\cdot\nn_i|},\\
E_h^{\bs\varphi} &:=& 
	\frac{\max_{i=1}^N |\bs\varphi_i^{\mathrm{eff}}-\bs U(\mm_i)|}
		{\max_{i=1}^N |\bs U(\mm_i)|}.
\end{eqnarray}
\end{subequations}
The expected convergence orders are
\[
E_h^{\bs U}=\mathcal O(N^{-3}), \qquad E_h^{\bs\lambda}=\mathcal O(N^{-3}),
\qquad E_h^{\bs\varphi}=\mathcal O(N^{-4}),
\]
where it has to be noted that $\bs\varphi$ is not obtained as a solution of an integral equation, but just projected from sampled data, which explains the higher order of convergence.  The errors are reported in Tables \ref{table:1} and \ref{table:2} and plotted in Figures \ref{figure:1} and \ref{figure:2}.

\begin{table}[h]\centering
\begin{tabular}{ccccccc}
\hline
\multicolumn{1}{|c|}{$N$} & \multicolumn{1}{c|}{$E^{\bs U}_{h}$} & \multicolumn{1}{c|}{e.c.r.} & \multicolumn{1}{c|}{$E^{\bs\lambda}_h$} & \multicolumn{1}{c|}{e.c.r.} & \multicolumn{1}{c|}{$E^{\bs\phi}_h$} & \multicolumn{1}{c|}{e.c.r.} \\ \hline
 30  	&  1.836   		  & ---   & 2.185   		     & ---   & 5.057 $\times 10^{-3}$  & ---           \\ \hline
 60 	&  3.338 $\times 10^{-2}$ & 5.761 & 2.682 $\times 10^{-2}$   & 6.348 & 1.212 $\times 10^{-4}$  & 5.382           \\ \hline
 120 	&  2.941 $\times 10^{-3}$ & 3.542 & 2.705 $\times 10^{-3}$   & 3.309 & 1.451 $\times 10^{-5}$  & 3.063           \\ \hline
 240 	&  4.352 $\times 10^{-4}$ & 2.757 & 3.224 $\times 10^{-4}$   & 3.069 & 1.011 $\times 10^{-6}$  & 3.843           \\ \hline
 480	&  5.640 $\times 10^{-5}$ & 2.948 & 3.983 $\times 10^{-5}$   & 3.017 & 6.484 $\times 10^{-8}$  & 3.963          \\ \hline
 960 	&  7.020 $\times 10^{-6}$ & 3.006 & 4.958 $\times 10^{-6}$   & 3.006 & 4.077 $\times 10^{-9}$  & 3.991           \\ \hline
 1920 	&  8.783 $\times 10^{-7}$ & 2.998 & 6.199 $\times 10^{-7}$   & 3.000 & 2.552 $\times 10^{-10}$ & 3.998             
\end{tabular}
\caption{Relative errors \eqref{eq:6.6} and estimated convergence rates in the frequency domain for $\bs U$, $\bs\lambda$ and $\bs\phi$ using equation \eqref{eq:6.4}.}\label{table:1}
\end{table}

\begin{table}[h]\centering
\begin{tabular}{ccccccc}
\hline
\multicolumn{1}{|c|}{$N$} & \multicolumn{1}{c|}{$E^{\bs U}_{h}$} & \multicolumn{1}{c|}{e.c.r.} & \multicolumn{1}{c|}{$E^{\bs\lambda}_h$} & \multicolumn{1}{c|}{e.c.r.} & \multicolumn{1}{c|}{$E^{\bs\phi}_h$} & \multicolumn{1}{c|}{e.c.r.} \\ \hline
 30  	&  3.317 $\times 10^{-1}$ & ---   & 5.806 $\times 10^{-1}$   & ---   & 5.057 $\times 10^{-3}$  & ---           \\ \hline
 60 	&  5.785 $\times 10^{-2}$ & 2.519 & 6.836 $\times 10^{-2}$   & 3.086 & 1.212 $\times 10^{-4}$  & 5.382           \\ \hline
 120 	&  7.271 $\times 10^{-3}$ & 2.993 & 7.620 $\times 10^{-3}$   & 3.165 & 1.451 $\times 10^{-5}$  & 3.063           \\ \hline
 240 	&  8.923 $\times 10^{-4}$ & 3.026 & 8.635 $\times 10^{-4}$   & 3.142 & 1.011 $\times 10^{-6}$  & 3.843           \\ \hline
 480	&  8.087 $\times 10^{-5}$ & 3.464 & 1.030 $\times 10^{-5}$   & 3.068 & 6.484 $\times 10^{-8}$  & 3.963          \\ \hline
 960 	&  1.004 $\times 10^{-5}$ & 3.009 & 1.256 $\times 10^{-6}$   & 3.035 & 4.077 $\times 10^{-9}$  & 3.991           \\ \hline
 1920 	&  1.473 $\times 10^{-6}$ & 2.769 & 1.605 $\times 10^{-7}$   & 2.969 & 2.552 $\times 10^{-10}$ & 3.998             
\end{tabular}
\caption{Relative errors \eqref{eq:6.6} and estimated convergence rates in the frequency domain for $\bs U$, $\bs\lambda$ and $\bs\phi$ using equation \eqref{eq:6.5}.}\label{table:2}
\end{table}

\begin{figure}[h]
\begin{center}
\includegraphics[width=9cm]{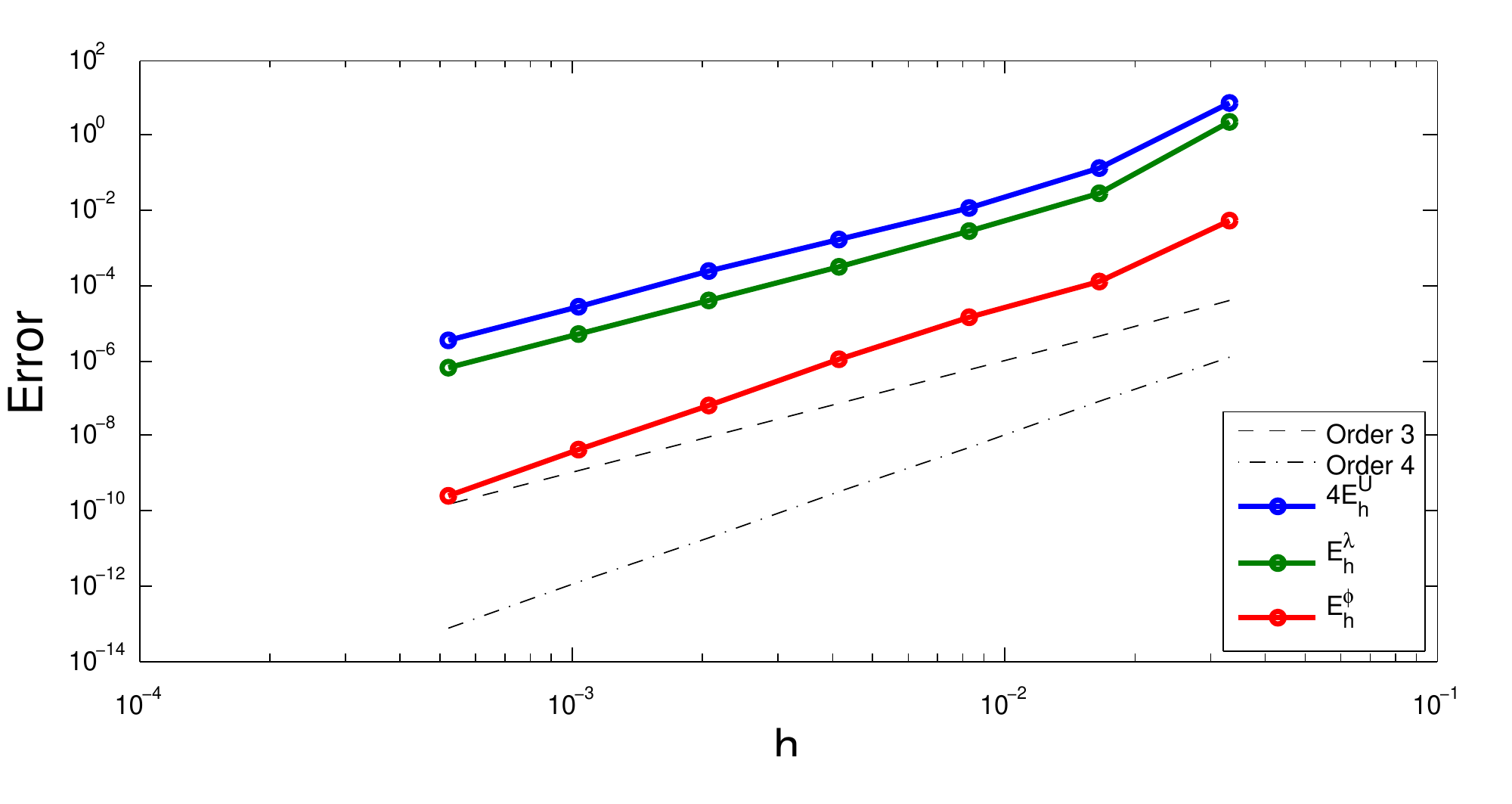}
\end{center}
\caption{Errors corresponding to Table \ref{table:1}. The error $E_h^{\bs U}$ has been rescaled to separate the error graphs.} \label{figure:1}
\end{figure}

\begin{figure}[h]
\begin{center}
\includegraphics[width=9cm]{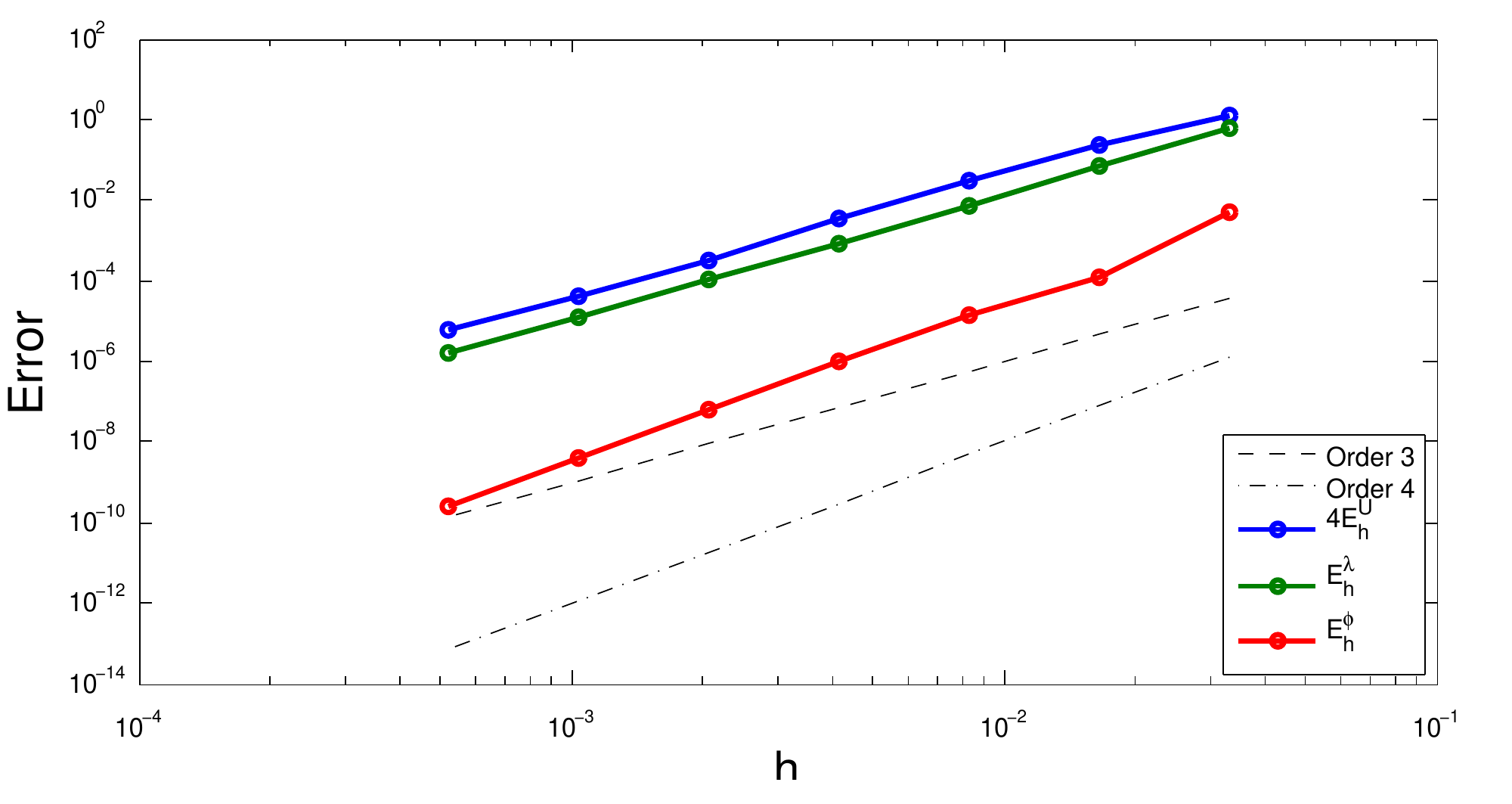}
\end{center}
\caption{Errors corresponding to Table \ref{table:2}. The error $E_h^{\bs U}$ has been rescaled to separate the error graphs.} \label{figure:2}
\end{figure}

\section{Experiments in the time domain}\label{sec:7}

In this section we show some tests on how to use the previously devised Calder\'on calculus for elastic waves in the Laplace/frequency domain to simulate transient waves by using a Convolution Quadrature method. We will be using an order two (BDF2-based) Convolution Quadrature strategy. For theoretical aspects of multistep CQ methods applied to wave propagation, the reader is referred to \cite{Lubich:1994, LaSa:2009, Sayas:2014}. Practical aspects on the implementation of CQ can be found in \cite{Banjai:2010, BaSc:2012, HaSa:2014}. A concise explanation on how to use CQ for an acoustic wave propagation problem discretized with a method that is similar to the one in this paper is given in \cite{DoLuSa:2014b}.  For the purposes of exposition we can present CQ as  two blackboxes. In the forward CQ, a sequence of vectors $\bs\beta^n\in \mathbb R^{n_1}$ (for $n\ge 0$) is input, a  $\mathbb C^{n_2\times n_1}$-valued transfer operator $\mathrm A(s)$ is given, and another sequence of vectors $\bs\delta^n\in \mathbb R^{n_2}$ is output. We will denote this as 
\[
\bs\delta^n=\mathrm{CQfwd}(\mathrm A(s),\bs\beta^n),
\]
with the implicit understanding that a particular time value $\bs\delta^n$ depends on the past of the input $\bs\beta^m$ for $m\le n$. 
We note that even if CQ can be understood and implemented as a time-stepping (marching-on-in-time) method, these small experiments are carried out using an all-times-at-once strategy \cite{BaSc:2012} involving evaluations of the transfer function $\mathrm A(s)$ at many different complex frequencies. The second kind of CQ blackbox takes input $\bs\beta^n\in \mathbb R^r$, uses a invertible transfer operator $\mathrm A(s)\in \mathbb C^{r\times r}$, and outputs another sequence $\bs\delta^n\in \mathbb R^r$. This convolution equation process is equivalent to the forward process applied to the inverse transfer function
\[
\bs\delta^n=\mathrm{CQeqn}(\mathrm A(s),\bs\beta^n)=\mathrm{CQfwd}(\mathrm A(s)^{-1},\bs\beta^n),
\]
and can be implemented in different ways, either by solving equations related to $\mathrm A(s)$ at many complex frequencies or by repeatedly inverting $\mathrm A(c_0)$ for a large positive value $c_0$ (which corresponds to a highly diffusive equation). 

\paragraph{Errors.}
We use the same geometry and physical parameters as in the experiments of Section \ref{sec:6}. We solve an interior problem with prescribed Dirichlet (resp. Neumann) boundary condition taken so that the exact solution is the plane pressure wave
\begin{equation}\label{eq:7.1}
\bs U(\mathbf z,t)=H(c_L(t-t_0)-\mathbf z\cdot\mathbf d) \sin(2 (c_L(t-t_0)-\mathbf z\cdot\mathbf d))\mathbf d,
\end{equation}
with $\mathbf d=(1/\sqrt2,1/\sqrt2)$ and $t_0=2.3.$ Here $H$ is a smoothened version of the Heaviside function. The problem is integrated from $0$ to $T=3$, on a space grid with $N$ points,  with $M$ time steps of length $T/M$: $t_n=n\,T/M$. Relative errors for the solution are computed at the observation points $\mathbf z_i^{\mathrm{obs}}$ (those of Section \ref{sec:6}). For the Dirichlet problem we use a single layer potential representation of the solution: using \eqref{eq:3.7}-\eqref{eq:3.8} we build sample vectors $\bs\beta_0^n$ from $\gamma\bs U(\cdot,t_n)$, we then solve the discrete time-domain integral equations
\[
\bs\eta^n=\mathrm{CQeqn}(\mathrm V_h(s),\bs\beta_0^n)
\]
and finally postprocess the discrete densities to compute approximations of the displacement field at the the observation points:
\[
\bs U_h^n(\mathbf z_i^{\mathrm{obs}})=\mathrm{CQfwd}(\mathrm S_h(s;\mathbf z_i^{\mathrm{obs}}),\bs\eta^n).
\]
For the Neumann problem we create sample vectors $\bs\beta_1^n$ from $\boldsymbol t(\bs U)(\cdot,t_n)$, solve the equations
\[
\bs\psi^n=-\mathrm{CQeqn}(\mathrm W_h(s),\bs\beta_1^n),
\]
and postprocess at the observation points
\[
\bs U_h^n(\mathbf z_i^{\mathrm{obs}})=\mathrm{CQfwd}(\mathrm D_h(s;\mathbf z_i^{\mathrm{obs}}),\bs\psi^n).
\]
In this experiment we fix a space discretization with $n=500$ points and refine in the number of time steps $M$. The results are reported in Table \ref{table:3}. The expected order two convergence of CQ is observed until the error due to space discretization dominates. Similar results refining in time and space can be produced in the same way.

\begin{table}[h]\centering
\begin{tabular}{ccccc}
\hline
\multicolumn{1}{|c|}{$M\phantom{\Big|}$} & \multicolumn{1}{c|}{$E^{\bs U}_{h}$ Dirichlet} & \multicolumn{1}{c|}{e.c.r.} & \multicolumn{1}{c|}{$E^{\bs U}_{h}$ Neumann} & \multicolumn{1}{c|}{e.c.r.} \\ \hline
 50  	&  2.828 $\times 10^{-1}$ & ---   & 2.676 $\times 10^{-1}$   & ---   \\ \hline
 100 	&  1.711 $\times 10^{-1}$ & 0.725 & 1.662 $\times 10^{-1}$   & 0.686 \\ \hline
 200 	&  5.181 $\times 10^{-2}$ & 1.724 & 5.205 $\times 10^{-2}$   & 1.675 \\ \hline
 400 	&  1.420 $\times 10^{-2}$ & 1.878 & 1.401 $\times 10^{-2}$   & 1.893 \\ \hline
 800	&  3.070 $\times 10^{-3}$ & 2.199 & 2.983 $\times 10^{-3}$   & 2.232 \\ \hline
 1600 	&  6.171 $\times 10^{-4}$ & 2.313 & 6.124 $\times 10^{-4}$   & 2.284             
\end{tabular}
\caption{Relative errors and estimated convergence rates in the time domain for the displacement at the final time $T = 3$ with  Dirichlet and Neumann boundary conditions. The first column shows the number of time steps, 500 discretization points in space were used. }\label{table:3}
\end{table}

\paragraph{An illustrative experiment.} We next show the capabilities of the method for the discretization of the scattering of an elastic wave by three obstacles. The obstacles are three disks with boundaries:
\[
(x-1)^2+(y-1)^2=1,
\qquad
(x-3)^2+(y-3)^2=1,
\qquad
(x-3.5)^2+(y-0.4)^2=1.
\]
For an incident wave $\bs U^{\mathrm{inc}}(\mathbf z,t)$, we use the pressure wave given in \eqref{eq:7.1}.
We look for a causal displacement field satisfying the elastic wave equation \eqref{eq:1.3} with boundary condition $\gamma\bs U+\gamma\bs U^{\mathrm{inc}}=\bs 0$ at all times. We use a direct formulation that we next explain in the Laplace domain. Data are sampled in space (with 200 points per obstacle) using \eqref{eq:3.7}-\eqref{eq:3.8} at equally spaced time-steps of length $k=28/1200$ and stored in vectors $\bs\beta_0^n$. These data have first to be projected on a space of traces $\bs\varphi^n=-\mathrm M^{-1}\bs\beta_0^n$ and then
run through a second hand integral operator in order to build the right-hand-side of the equations,
\[
\bs\xi^n=-\smallfrac12 \mathrm M\bs\varphi^n+\mathrm{CQfwd}(\mathrm K_h(s),\bs\varphi^n).
\]
(Compare with \eqref{eq:6.2} and \eqref{eq:6.4} and note the different sign due to the fact that we are solving an exterior problem.)
An approximation of the normal stresses is then computed by solving
\[
\bs\lambda^n=\mathrm{CQeqn}(\mathrm V_h(s),\bs\xi^n).
\]
Finally the solution is evaluated at a large number of observation points $\mathbf z_i$ using two potentials
\[
\bs U^n(\mathbf z_i)=\mathrm{CQfwd}(\mathrm D_h(s;\mathbf z_i),\bs\varphi^n)-
					\mathrm{CQfwd}(\mathrm S_h(s;\mathbf z_i),\bs\lambda^n).
\]
We show four snapshots of the solution, which is asymptotically reaching the time-harmonic regime. We plot the absolute value of the displacement. The scattered pressure and shear waves can be distinguished by the different speeds of propagation that appear already after the wave hits the first obstacle.

\begin{figure}[h]
\begin{center}
\begin{tabular}{cc}
\includegraphics[width=7cm]{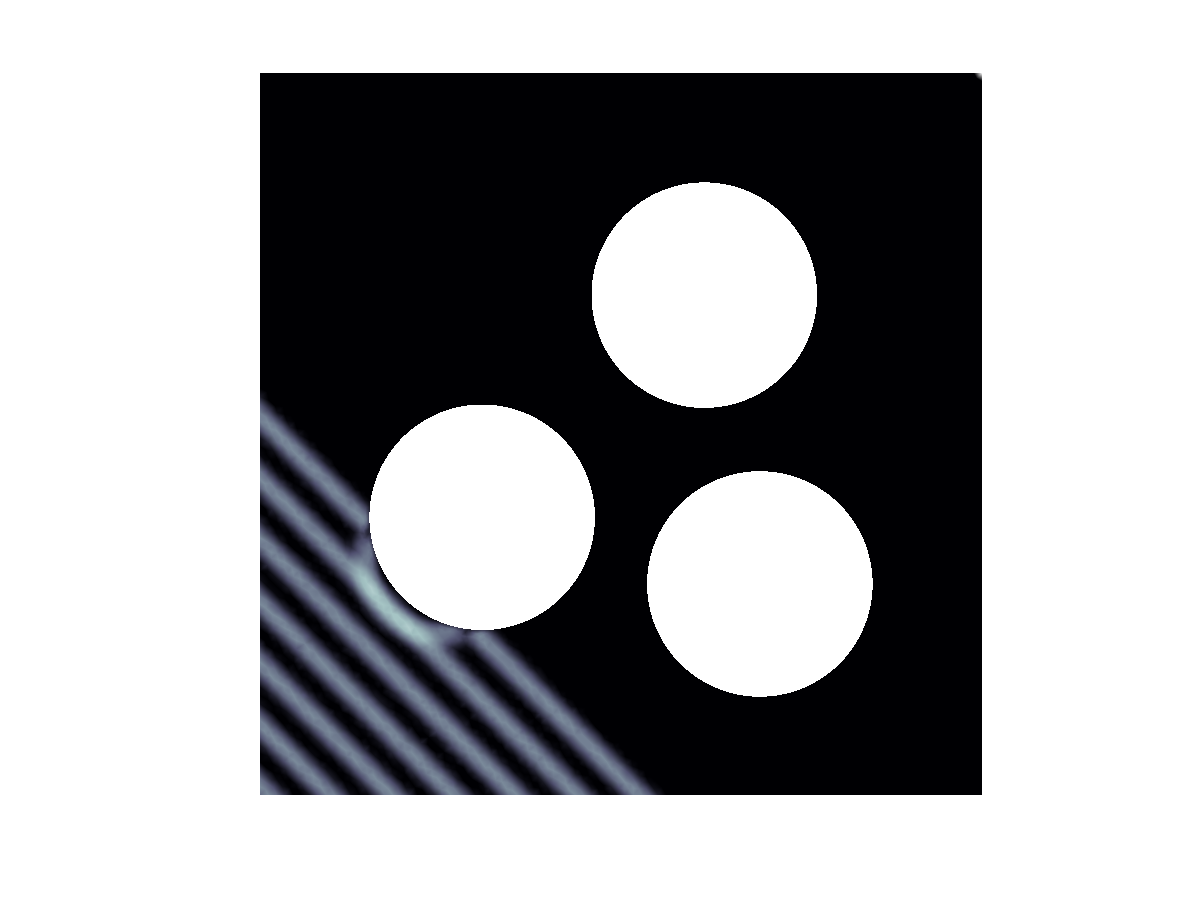} & \includegraphics[width=7cm]{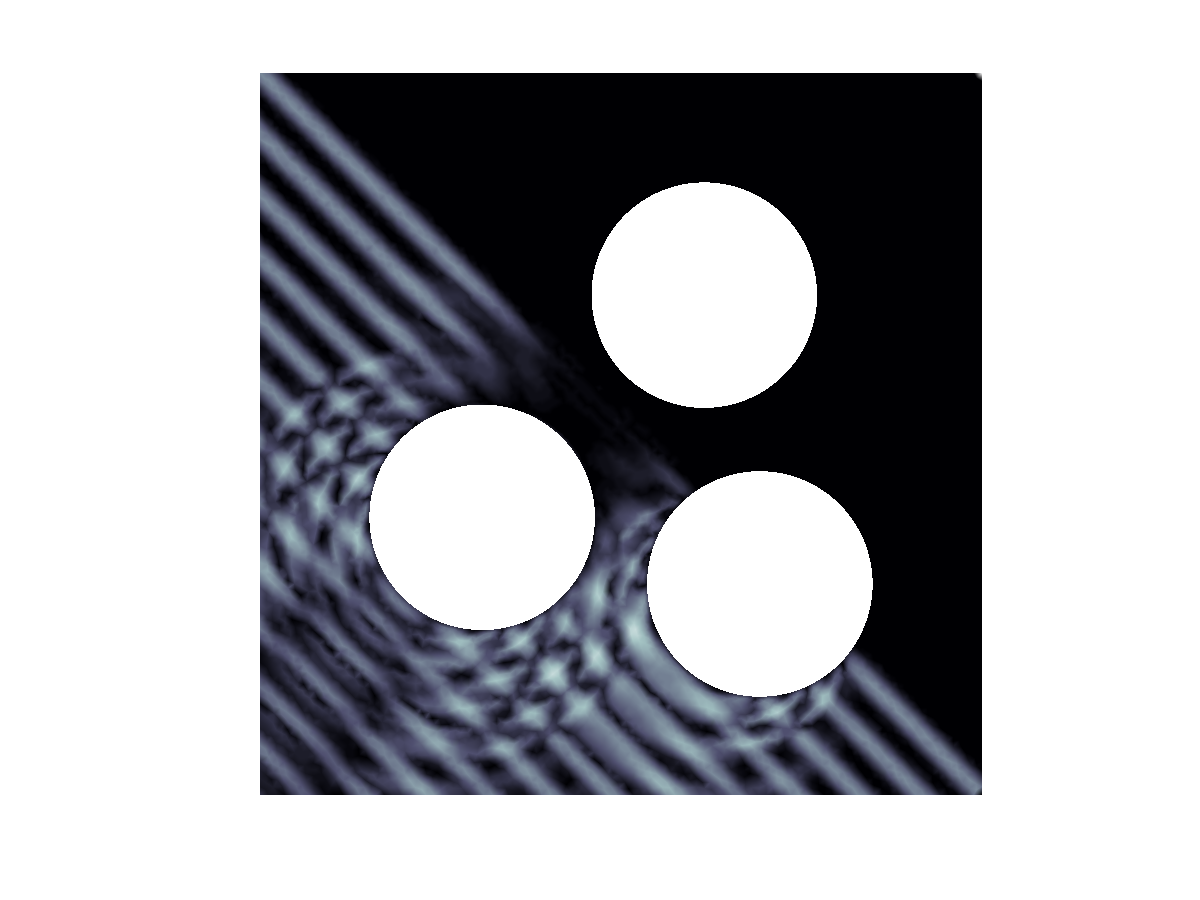} \\
\includegraphics[width=7cm]{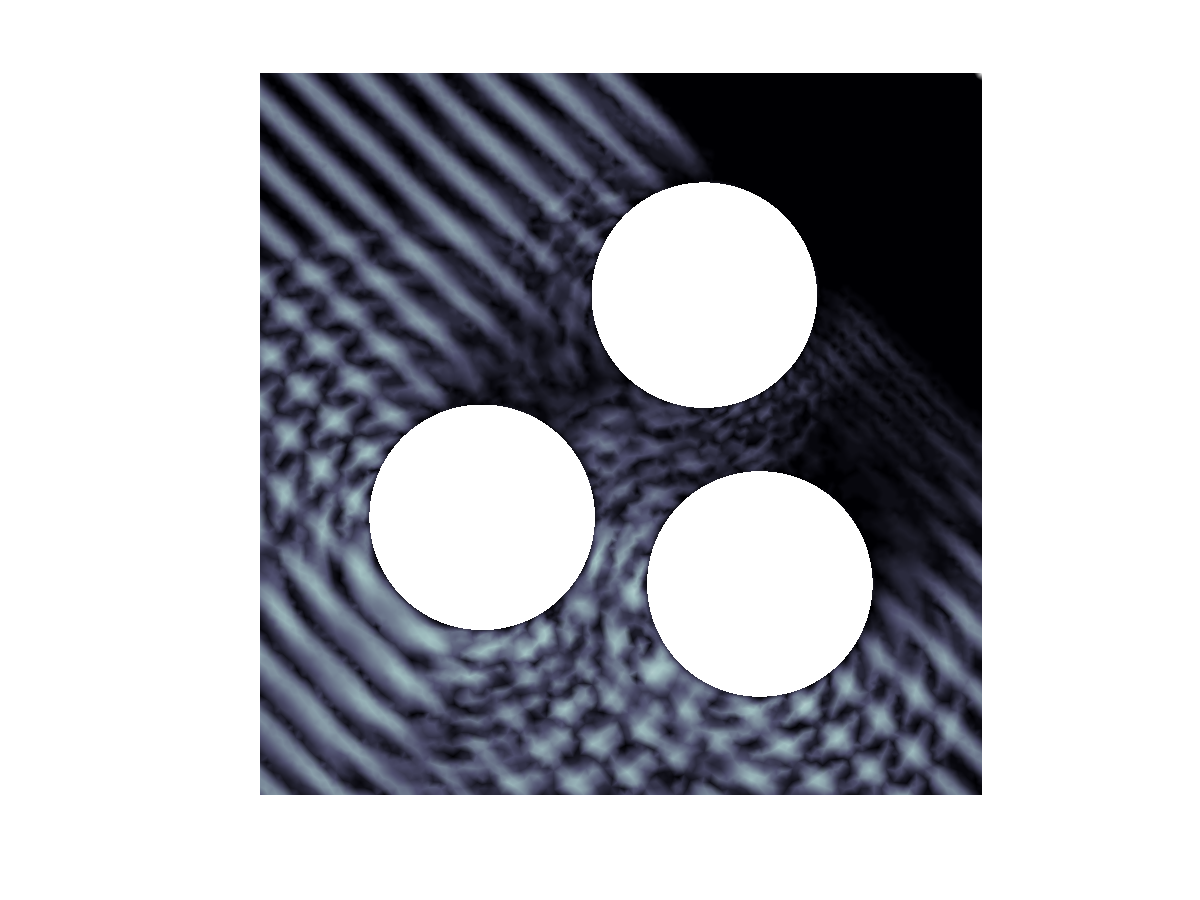} & \includegraphics[width=7cm]{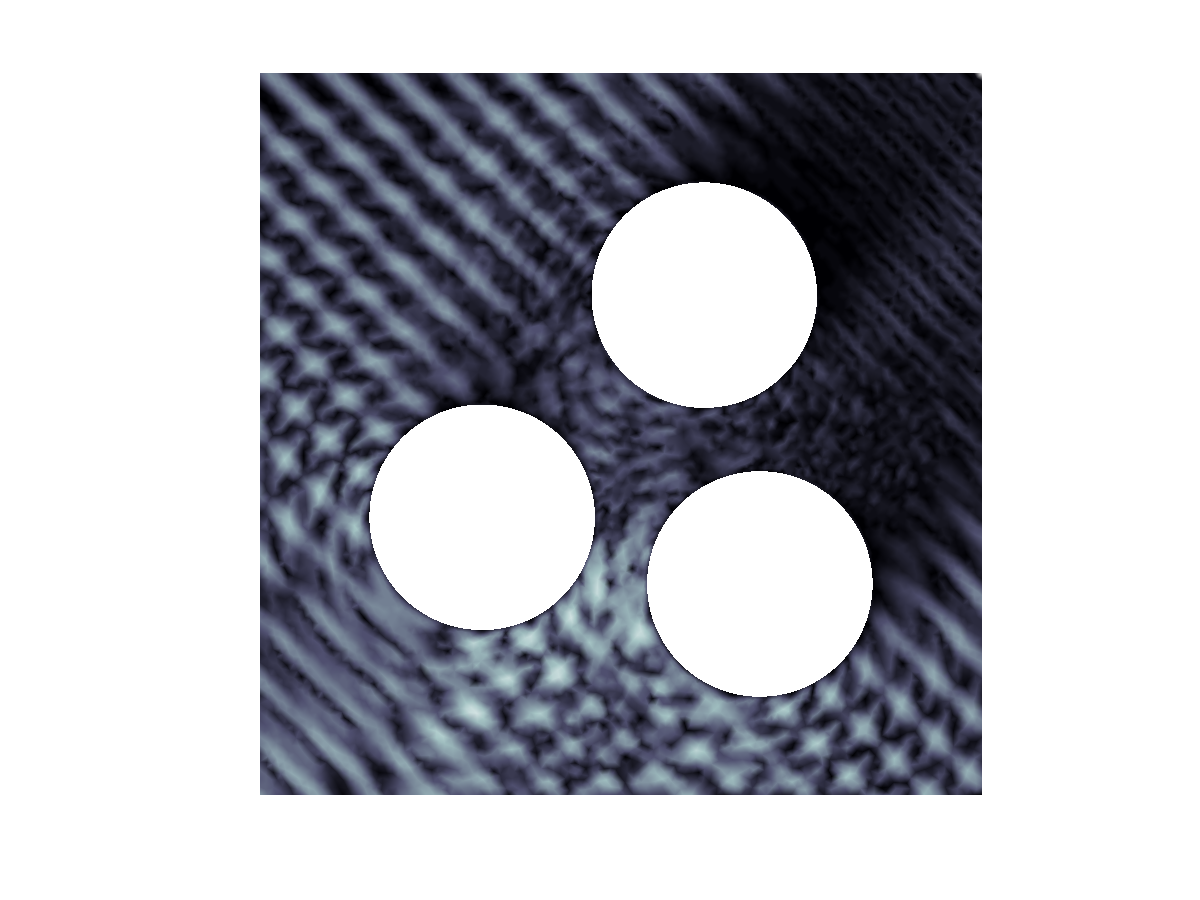} 
\end{tabular}
\end{center}
\caption{Four images of the scattering of a plane pressure wave by three rigid obstacles. The solution transitions to a time-harmonic regime. The absolute value of the displacement field is shown in a gray scale (black is no-displacement).}\label{figure:3}
\end{figure}

\section{The treatment of open arcs}\label{sec:8}

In the same spirit as the previous sections, smooth open arcs can be incorporated into the fully discrete Calder\'on Calculus for elastic waves. The idea is simple: instead of using a traditional cosine change of variable to modify the integral equation into a periodic integral equation, we will use the cosine change of variables to sample geometric features from the open arc and then define the discrete elements (the two layer potentials and the operators $\mathrm V$ and $\mathrm W$) using the same formulas as in the case of closed curves. Note that the operators $\mathrm K$ and $\mathrm J$ are not meaningful in the case of open arcs.

\paragraph{Cosine sampling of open arcs.} Let $\xx:[0,1]\to\mathbb R^2$ be a regular parametrization of a smooth simple open arc. Let us also consider the 1-periodic even function $\phi:\mathbb R\to[0,1]$ given by 
\[
\phi(t):=\smallfrac12 +\smallfrac12\cos(\pi(2t-1)).
\]
We then define
$\aaa(t):=\xx(\phi(t))$ and note that $\aaa(0)=\xx(0)=\aaa(1)$, $\aaa(\frac12)=\xx(1)$, and $\aaa(1-t)=\aaa(t)$ for all $t$. The normal vector field
\[
\nn(t):=\aaa'(t)^\perp=-\pi \sin((\pi(2t-1))\, \xx'(\phi(t))^\perp, \quad\mbox{where}\quad (c_1,c_2)^\perp:=(c_2,-c_1),
\]
satisfies $\nn(0)=\nn(\frac12)=\bs 0$, and $\nn(1-t)=-\nn(t)$ for all $t$. This means that the normal vector has different signs depending on whether we are moving from the first tip  to the second or back. We now choose a positive even integer $N=2M$, define $h:=1/N=1/(2M)$, and create the main discrete grid:
\begin{equation}\label{eq:8.1}
\mm_j:=\aaa((j-\smallfrac12)h), 
\qquad
\bb_j:=\aaa((j-1)h),
\qquad
\nn_j:=h\nn((j-\smallfrac12)h), \qquad j\in \mathbb Z_{N}.
\end{equation}
Let us first comment on these formulas. In comparison with \eqref{eq:2.1} the breakpoints and midpoints are displaced $\frac12h$ in parametric space. Whereas in the case of closed curves this is not relevant, for open arcs it will be essential that the ends of the arc are breakpoints of the discrete grid: $\bb_1=\bb_{2M+1}=\xx(0)$, $\bb_{M+1}=\xx(1)$. Also, note that points are sampled twice and we actually have:
\[
\mm_{N+1-j}=\mm_j, \qquad \bb_{N+2-j}=\bb_j, \qquad \nn_{N+1-j}=-\nn_j.
\]
The companion grids are similarly collected from the curve:
\begin{equation}\label{eq:8.2}
\mm_j^\pm:=\aaa((j-\smallfrac12 \pm \smallfrac16)h),
\quad
\bb_j^\pm:=\aaa((j-1\pm\smallfrac16)h),
\quad
\nn_j^\pm:=h \nn((j-\smallfrac12\pm \smallfrac16)h),
\quad j\in \mathbb Z_{N}.
\end{equation}

\paragraph{Even and odd vectors.} The duplication in the cosine sampling will have some effects in the structure of the discrete Calder\'on Calculus. Let $\bs\xi\in \mathbb C^{2N}$ be a vector with blocks $\bs\xi_j$. We will write $\bs\xi\in \mathbb C^{2N}_{\mathrm{even}}$ when $\bs\xi_{N+1-j}=\bs\xi_j$ for all $j$, and we will write
$\bs\xi\in \mathbb C^{2N}_{\mathrm{odd}}$ when $\bs\xi_{N+1-j}=-\bs\xi_j$ for all $j$. Even vectors are those in the nullspace of the matrix
\begin{equation}\label{eq:8.3}
\mathrm H:=
	\left[\begin{array}{cc} 
		\mathrm H_{\mathrm{sc}} & 0 \\
		0 & \mathrm H_{\mathrm{sc}}
	\end{array}\right],
\qquad
\mathrm H_{\mathrm{sc}}:=
	\left[\begin{array}{cccccc} 
		1 & & & & & -1 \\
		   & \ddots & & & \iddots \\
		  & & 1 & -1 & & \\
		 & & -1 & 1 & & \\
		 & \iddots & & & \ddots \\
		-1 & & & & & 1
	\end{array}\right],
\end{equation}
while odd vectors are in the nullspace of the matrix $|\mathrm H|$, obtained by taking the absolute value of the elements of $\mathrm H$. Additionally $\mathrm H\bs\xi=2\bs\xi$ if $\bs\xi\in \mathbb C^{2N}_{\mathrm{odd}}$ and $|\mathrm H|\bs\xi=2\bs\xi$ if $\bs\xi\in \mathbb C^{2N}_{\mathrm{even}}$. Finally if we sample an incident wave on an open arc using \eqref{eq:3.7}-\eqref{eq:3.8}, it is not difficult to prove that
\[
\bs\beta_0\in \mathbb C^{2N}_{\mathrm{even}} 
\qquad \mbox{and}\qquad
\bs\beta_1\in \mathbb C^{2N}_{\mathrm{odd}}.
\]

\paragraph{A Dirichlet crack.} If $\Gamma$ is the smooth open curve given in parametric form at the beginning of this section, we look for solutions of $\rho \partial_t^2 \bs U= \mathrm{div}\,\bs\sigma(\bs U)$ in $\mathbb R^2\setminus\Gamma$, with the corresponding radiation condition at infinity (see the comments after formula \eqref{eq:1.4}), and the Dirichlet condition $\gamma\bs U+\gamma\bs U^{\mathrm{inc}}=\bs 0$ on $\Gamma$. We assume that the discrete data have already been sampled with \eqref{eq:8.1}-\eqref{eq:8.2} and that the incident wave has been observed using \eqref{eq:3.7}-\eqref{eq:3.8}, outputting a vector $\bs\beta_0\in \mathbb C^{2N}_{\mathrm{even}}$. We now use the fact that
\[
\mathrm V_h(s)\bs\eta\in \mathbb C^{2N}_{\mathrm{even}} \quad \forall \bs\eta\in \mathbb C^{2N},
\quad \mbox{and}\quad
\mathrm V_h(s)\bs\eta=\bs 0 \quad\forall \bs\eta\in \mathbb C^{2N}_{\mathrm{odd}}.
\]
(These properties are not difficult to prove.) Therefore, if we solve
\[
(\mathrm V_h(s)+\mathrm H)\bs\eta=-\bs\beta_0,
\]
we are guaranteed that $\bs\eta\in \mathbb C^{2N}_{\mathrm{even}}$. This density is used to build the discrete elastic wavefield $\bs U_h(\mathbf z)=\mathrm S_h(s;\mathbf z)\bs\eta$. 

\paragraph{A Neumann crack.} In this case, we substitute the Dirichlet boundary condition on the open curve $\Gamma$ by $\bs t (\bs U)+\bs t(\bs U^{\mathrm{inc}})=\bs 0$. Using a similar argument as in the Dirichlet case, we solve the discrete integral equation
\[
(\mathrm W_h(s)+|\mathrm H|)\bs\psi=\bs\beta_1,
\]
obtain a unique $\bs\psi\in \mathbb C^{2N}_{\mathrm{odd}}$ and input it in a double layer potential representation $\bs U_h(\mathbf z)=\mathrm D_h(s;\mathbf z)\bs\psi$. 

\paragraph{Experiments.} We use the half circle $\Gamma=\{x^2+y^2=1\}\cap \{ y\ge0\}$ as the scattering arc. The physical parameters of the surrounding unbounded elastic medium are those of Section \ref{sec:6}. The incident wave is the pressure part of the function defined in \eqref{eq:6.0}. The solution is observed at ten random points on a circle of radius $5$. Since the exact solution is not know, we use a three grid principle to estimate convergence rates.

\begin{table}[h]\centering
\begin{tabular}{ccc}
\hline
\multicolumn{1}{|c|}{$N$} & \multicolumn{1}{c|}{e.c.r. Dirichlet Problem} & \multicolumn{1}{c|}{e.c.r. Neumann Problem} \\ \hline
 10  	&  ---   & ---   \\ \hline
 20  	&  ---   & ---   \\ \hline
 40  	&  3.449 & 4.446   \\ \hline
 80 	&  3.120 & 3.512 \\ \hline
 160 	&  3.030 & 3.318 \\ \hline
 320 	& 3.008  & 3.010          
\end{tabular}
\caption{Estimated convergence rates in the frequency domain Dirichlet and Neumann cracks. Given the fact that the solution is not known, a three grid principle is used to estimate convergence rates.}
\end{table}

\appendix

\section{Some functions}

For the sake of completeness we give here explicit expressions of all the functions that are involved in the definitions of the potentials and integral operators. First of all, we give the derivatives of the functions \eqref{eq:3.0}:
\begin{eqnarray*}
\psi'(r) &=& -\frac1{c_T} K_1(r/c_T)
			-\frac{2c_T}{r^2}\left(K_1(r/c_T)-\xi K_1(r/c_L)\right)
			-\frac1r\left( K_0(r/c_T)-\xi^2 K_0(r/c_L)\right),\\
\chi'(r) &=& -\frac1{2c_T} \left(K_1(r/c_T)+K_3(r/c_T)-\xi^3 \left( K_1(r/c_L)+K_3(r/c_L)\right)\right).
\end{eqnarray*}
The functions \eqref{eq:4.5} are combinations of the following derivatives:
\begin{eqnarray*}
G'(r) & = & \frac{-1}{2\pi\rho c_T}\big(K_1(r/c_T)-\xi K_1(r/c_L)\big) \\
G''(r) & = &  \frac{1}{4\pi\rho c_T^2}\big(K_0(r/c_T)+K_2(r/c_T)-\xi^2(K_0(r/c_L)+K_2(r/c_L))\big) \\
G'''(r) & = & \frac{-1}{8\pi\rho c_T^3}\big(3K_1(r/c_T)+K_3(r/c_T)-\xi^3(3K_1(r/c_L)+K_3(r/c_L))\big) \\
G^{(iv)}(r) & = &  \frac{1}{2\pi\rho c_T^4}\left(\left(\frac{3c_T^2}{r^2}+ 1\right)K_2(r/c_T)-\xi^4\left(\frac{3c_L^2}{r^2}+ 1\right)K_2(r/c_L)\right)
\end{eqnarray*}

\section{A technical identity}\label{app:B}

In the following results, for a given positive integer $N$ we will write $h:=1/N$ and $t_j:=j\,h$ for $j\in \mathbb Z$. We also let $\log_\# t:=\log(4\sin^2(\pi t))$. 

\begin{lemma}
For all $N$ and $t$
\begin{equation}\label{eq:B.1}
\sum_{j=1}^N \log_\#(t-t_j)=\log_\#(t/h), \qquad t/h\not\in \mathbb Z.
\end{equation}
\end{lemma}

\begin{proof}
We recall the Fourier expansion of the periodic logarithm
\begin{equation}\label{eq:B.2}
\log_\# t=\lim_{M\to\infty}\sum_{0\neq n=-M}^N \frac1{|n|}e_n(t), \qquad e_n(t):=\exp(2\pi\imath n t).
\end{equation}
Also
\begin{equation}\label{eq:B.3}
\sum_{j=1}^N e_n(t-jh)=e_n(t)\sum_{j=1}^N \exp\left(-\frac{2\pi\imath j n}{N}\right)=
e_n(t)\left\{\begin{array}{ll} N, & \mbox{if $n/N \in \mathbb Z$}, \\ 0, & \mbox{otherwise}.\end{array}\right.
\end{equation} 
Combining \eqref{eq:B.2} and \eqref{eq:B.3} it is easy to see that
\begin{eqnarray*}
\sum_{j=1}^N \log_\# (t-jh)
	&=& \lim_{M\to \infty} \sum_{0\neq n=-M}^M \frac1{|nN|}N e_{nN}(t)\\
	&=& \lim_{M\to \infty} \sum_{0\neq n=-M}^M  \frac1{|n|} e_n(t/h)=\log_\#(t/h).
\end{eqnarray*}
\end{proof}

\begin{lemma}\label{lemma:B.2}
For all $N\in \mathbb Z$ and $|n|\le N-1$,
\[
h\sum_{j=1}^{N-1}\cot(\pi(t-t_j))e_n(t_j)=\mathrm{p.v.}\int_0^1 \cot(\pi(t-\tau))e_n(\tau)\mathrm d\tau
+\cot(\pi t/h) e_n(t), \quad t/h\not\in\mathbb Z.
\]
\end{lemma}

\begin{proof}
Differentiating \eqref{eq:B.1} we obtain
\begin{equation}\label{eq:B.4}
h\sum_{j=1}^{N-1}\cot(\pi(t-t_j))=\cot(\pi t/h), \qquad t/h\not\in \mathbb Z,
\end{equation}
which is the case $n=0$. Consider now the function
\[
s_n(\tau;t):=\cot(\pi(t-\tau))(e_n(\tau)-e_n(t))=e_n(t)\cot(\pi(\tau-t))(1-e_n(\tau-t)).
\]
Note now that for $n \ge 1$,
\begin{eqnarray*}
\imath \cot(\pi\tau) (1-e_{n}(\tau))&=&(1+e_1(\tau)) \frac{1-e_{n}(\tau)}{1-e_{1}(\tau)} \\
 &=&(1+e_{1}(\tau)) (1+e_{1}(\tau)+e_{2}(\tau)+ \ldots+e_{n-1}(\tau))\\
  &=&  1+2e_{1}(\tau)+2e_{2}(\tau)+ \ldots+2e_{n-1}(\tau)+e_{n}(\tau)
 \end{eqnarray*}
and therefore
\[
s_n(\cdot;t)\in \mathrm{span}\{ e_m \,:\, 0\le m\le n\}\qquad n\ge 1. 
\]
By conjugation, it is easy to see then that 
\begin{equation}\label{eq:B.5}
s_n(\cdot;t)\in \mathbb T_{N-1}:=\mathrm{span}\{ e_m\,:\, |m|\le N-1\}, \quad |n|\le N-1.
\end{equation}
It is well known that
\[
\int_0^1\phi(\tau)\mathrm d\tau=h\sum_{j=1}^N \phi(t_j) \qquad \forall \phi\in \mathbb T_{N-1}.
\]
Therefore, by \eqref{eq:B.5},
\begin{eqnarray*}
\mathrm{p.v.}\int_0^1 \cot(\pi(t-\tau))e_n(\tau)\mathrm d\tau
	&=&\int_0^1 s_n(\tau;t)\mathrm d\tau = h\sum_{j=1}^N s_n(t_j;t) \\
	&=& h\sum_{j=1}^N \cot(\pi(t-t_j))e_n(t_j)-e_n(t) h \sum_{j=1}^N \cot(\pi(t-t_j))\\
	&=&h\sum_{j=1}^N \cot(\pi(t-t_j))e_n(t_j)-e_n(t)\cot(\pi t/h),
\end{eqnarray*}
where in the last equality we have used \eqref{eq:B.4}.
\end{proof}

\bibliographystyle{abbrv}
\bibliography{referencesW}

\end{document}